\documentclass[12pt]{amsart}
\usepackage{amsmath,amssymb,latexsym, amsthm, amscd, mathrsfs, stmaryrd, tikz, enumitem}

\usepackage[linktocpage=true]{hyperref}
\usepackage{color}
\usepackage[all]{xy}

\newtheorem{thm}{Theorem} [section]
\theoremstyle{definition}

\newtheorem{rem}[thm]{Remark}
\theoremstyle{plain}
\newtheorem{prop}[thm]{Proposition}
\newtheorem{lem}[thm]{Lemma}

\numberwithin{equation}{section}

\newcommand{\al}\alpha
\newcommand{\bbb}{\mathfrak{b}}

\newcommand{\C}{\mathbb C}
\newcommand{\D}{D(2|1;\zeta)}
\newcommand{\Dist}{\text{Dist}}
\newcommand{\DG}{\text{Dist}(G)}
\newcommand{\DGev}{\text{Dist}(G_{\bar 0})}

\newcommand{\de}{\delta}
\newcommand{\ep}{\epsilon}
\newcommand{\ve}{\epsilon}

\newcommand{\hf}{{\Small \frac12}}

\newcommand{\g}{\mathfrak{g}}
\newcommand{\ga}{\gamma}

\newcommand{\h}{\mathfrak{h}}
\newcommand{\ka}{\zeta}

\newcommand{\la}{\lambda}

\newcommand{\N}{\mathbb N}
\newcommand{\one}{{\ov 1}}
\newcommand{\osp}{\text{osp}}
\newcommand{\om}\omega
\newcommand{\ov}{\overline}
\newcommand{\OO}{\mathcal O}

\newcommand{\sll}{\mathfrak{sl}_2}

\newcommand{\Z}{\mathbb Z}
\newcommand{\oo}{{\ov 0}}
\newcommand{\bbc}{\mathbb{C}}

\newcommand{\Lie}{\mathsf{Lie}}
\newcommand{\res}{\mathsf{res}}
\newcommand{\ind}{{\mathsf{ind}}}
\newcommand{\Gev}{G_\oo}
\newcommand{\Hev}{H_\oo}
\newcommand{\Bev}{B_\oo}
\newcommand{\ch}{\mathsf{ch}}

\newcommand{\red}[1]{{\color{red}#1}}

\newcommand{\magenta}[1]{{\color{magenta}#1}}

\title{Modular representations of exceptional supergroups}

\author[Shun-Jen Cheng]{Shun-Jen Cheng}
\address{Institute of Mathematics, Academia Sinica, Taipei, Taiwan 10617} \email{chengsj@gate.sinica.edu.tw}

\author[Bin Shu]{Bin Shu}
\address{Department of mathematics, East China Normal University, Shanghai, China 200241}
\email{bshu@math.ecnu.edu.cn}

\author[Weiqiang Wang]{Weiqiang Wang}
\address{Department of Mathematics, University of Virginia, Charlottesville, VA 22904} \email{ww9c@virginia.edu}

\keywords{Exceptional supergroups, simple modules, odd reflections.}
\subjclass[2010]{Primary 20G05, 17B25}

\begin{document}

\maketitle

\begin{abstract}
We classify the simple modules of the exceptional algebraic supergroups over an algebraically closed field of prime characteristic.
\end{abstract}

\maketitle


 \section*{Introduction}

Among the simple Lie superalgebras over the complex field $\C$, the basic Lie superalgebras distinguish themselves by admitting a non-degenerate super-symmetric even bilinear form (see, e.g.,  \cite{CW12}), and they include 3 exceptional Lie superalgebras: $\D, G(3)$ and $F(3|1)$; cf. ~\cite{FK76}. The classification of finite-dimensional simple modules of complex simple Lie superalgebras was achieved by Kac \cite[Theorem 8]{Kac77}. Note that the simple highest weight modules whose highest weights are dominant integral (with respect to the even subalgebra) are not all finite dimensional. This is one of several aspects that super representation theory differs from the classical representation theory dramatically. This classification theorem of Kac can be reformulated as a classification for simple modules over the corresponding supergroups over $\C$.

There are algebraic supergroups associated to the basic Lie superalgebras, valid over an algebraically closed field $k$ of prime characteristic $p\neq 2$. A general theory of Chevalley supergroups was systematically developed by Fioresi and Gavarini \cite{FG12} (also see \cite{G14}).
In representation theory of algebraic supergroups $G$ over $k$, one of the basic questions is to classify the simple $G$-modules. For type $A$, the answer is immediate as it is the same as for the even subgroup $G_\oo$.
For type $Q$ such a classification was obtained in \cite{BK03}, and it has applications to classification of simple modules of spin symmetric groups over $k$.
For type $\frak{osp}$, the classification was obtained in \cite{SW08} in terms of the Mullineux involution by using odd reflections; also see Remark~\ref{rem:osp}.

\vspace{3mm}

The goal of this paper is to classify  the simple $G$-modules, when $G$ is a simply connected supergroup of exceptional type. We shall assume throughout the paper that $p>2$ for $\D$ and $p>3$ for $G(3)$ or $F(3|1)$
(except in \S\ref{sec:p=3}). Under these assumptions, their corresponding supergroups admit non-degenerate super-symmetric even bilinear forms. We treat $G(3)$ for $p=3$  in \S\ref{sec:p=3}.

\vspace{3mm}

Let us outline the approach of this paper. An equivalence of categories (\cite{SW08}; also cf. \cite{MS17}) reduces the classification of simple $G$-modules to the classification of the highest weights of finite-dimensional simple modules $L(\la) =L^{\bbb} (\la)$ over the distribution superalgebra $\DG$, where $\bbb$ is the standard Borel subalgebra. We then reduce the verification of finite-dimensionality of $L(\la)$ to verifying that $L(\la)$ is locally finite over its {\it even} distribution subalgebra. The local finiteness criterion for $L(\la)$  is finally established by means of odd reflections (see \cite{LSS}), and is based on the following observation which seems to be well known to experts (see \cite{Se11}):

{\em For every positive even root $\alpha$ in the standard positive system, either $\alpha/2$ (if it is a root) or $\alpha$ appears as a simple root in some simple system $\Pi'$ associated to some $\bbb'$, where $\bbb'$ is a Borel subalgebra obtained via a sequence of odd reflections from $\bbb$.}

For the exceptional Lie superalgebras, we make this observation explicit in this paper. We compute the highest weight $L^{\bbb '} (\la ')$ for all possible Borel subalgebras $\bbb'$ as mentioned above. Requiring $\la'$ to be dominant integral for all possible $\bbb'$ gives the local finiteness criterion for $L(\la)$.

Recently, an approach to obtain characters of projective and simple modules in the BGG category $\OO$ for the exceptional Lie superalgebras over $\C$ has been systematically developed; see \cite{CW17} for $\D$. Building on this and the current work, one may hope to better understand the characters of projective and simple modules of the exceptional supergroups over a field of prime characteristic in the future.

\vspace{3mm}
The organization of this paper is as follows.
In Section~\ref{sec:general}, we review the equivalence between the category of finite-dimensional modules over a supergroup $G$ and the category of finite-dimensional $(\DG, T)$-modules, where $T$ is a maximal torus of $G$. We develop a criterion for the finite-dimensionality of simple $\DG$-modules $L(\la)$ via odd reflections.
We also review the formula for the Euler characteristic, which implies that a $\DG$-module $L(\la)$, with $\la$ dominant integral and $\la+\rho$ is regular, is always finite dimensional.

In Section~\ref{sec:D}, we analyze the highest weight constraints given by odd reflections of a simple finite-dimensional $\DG$-module when $G$ is of type $\D$. Here $\D$ is a family depending on a parameter $\zeta \in k\backslash \{0,-1\}$. We then classify the simple $G$-modules in Theorem~\ref{thm:D}.

In Section~\ref{sec:G}, we analyze the highest weight constraints given by odd reflections of a simple finite-dimensional $\DG$-module when $G$ is of type $G(3)$. We then classify the simple $G$-modules in Theorem~\ref{thm:G3}.

In Section~\ref{sec:F}, we study the supergroup $G$ of type $F(3|1)$. When the highest weight $\la =a\om_1 +b\om_2 +c\om_3+ d\om_4$ with $a, b, c\in \N, d\ge 4$ is dominant integral, the weight $\la +\rho$ is regular and hence the Euler character formula implies that the $\DG$-module $L(\la)$ is finite dimensional. For $d\leq 3$, it is rather involved to analyze the highest weight changes under sequences of odd reflections and formulate sufficient and necessary conditions for $L(\la)$ to be finite dimensional. We finally classify the simple $G$-modules in Theorem~\ref{thm:F4}.

Finally we remark that, although in this article we deal with an algebraically closed field of positive characteristic, the results also make sense in characteristic zero and give the known classification in this case; cf.~\cite{Kac77, Ma14}.

 \vspace{.4cm}
 {\bf Acknowledgment.}
S.-J.C. is partially supported by a MoST and an Academia Sinica Investigator grant; B.S. is partially supported by the National Natural Science Foundation of China (Grant Nos. 11671138, 11771279) and Shanghai Key Laboratory of PMMP (No. 13dz2260400); W.W. is partially supported by an NSF grant DMS-1702254. We thank East China Normal University and Institute of Mathematics at Academia Sinica for hospitality and support.

 \section{Modular representations of algebraic supergroups}
\label{sec:general}

\subsection{Algebraic supergroups and $(\DG,T)$-$\mathfrak{mod}$}

Throughout the paper, the ground field $k$ is assumed to be algebraically closed and of characteristic $p>2$ (sometimes we will specify a stronger assumption $p>3$).

We shall review briefly some generalities on algebraic supergroups; cf. \cite{BK03, SW08, FG12, MS17}. An (affine) algebraic supergroup $G$ is an affine superscheme whose coordinate ring $k[G]$ is a Hopf superalgebra that is finitely generated as a $k$-algebra, and gives rise to a functor from the category of commutative $k$-superalgebras to the the category of groups. The underlying purely even group $\Gev$ is a closed subgroup of $G$ corresponding to the Hopf ideal generated by $k[G]_{\overline 1}$, and it is an algebraic group in the usual sense.
For an algebraic supergroup $G$, the distribution superalgebra $\Dist(G)$, which is by definition the restricted dual of the Hopf superalgebra $k[G]$, is a cocommutative Hopf superalgebra.

We denote by $G$-$\mathfrak{mod}$ the category of rational $G$-modules with (not necessarily homogeneous) $G$-homomorphisms. Note that a $G$-module is always {\em locally finite}, i.e., it is a sum of finite-dimensional $G$-modules. Given a closed subgroup $T$ of $G$, a $\DG$-module $M$ is called a $(\DG,T)$-module if $M$ has a structure of a $T$-module such that the $\Dist(T)$-module structure on $M$ induced from the actions of $\DG$ and of $T$ coincide. We denote by $(\DG,T)$-$\mathfrak{mod}$ the category of locally finite $(\DG,T)$-modules, and denote by $\DG$-$\mathfrak{mod}$ the category of locally finite $\DG$-modules. (We shall always take $T$ to be a maximal torus of $G$ when $G$ is of basic type.)

\subsection{Modules of basic algebraic supergroups}

Let $\g$ be a basic Lie superalgebra over $k$ \cite{CW12, FG12, G14}, including the three exceptional types:  $\D$,  $G(3)$, and $F(3|1)$. The non-degenerate bilinear form $(\cdot, \cdot)$ of $\g$ over $k$ exists when the characteristic $p$ of $k$ satisfies $p>2$ for type $\mathfrak{gl}, \mathfrak{osp}$ and $\D$, and $p>3$ for $G(3)$ and $F(3|1)$.

Algebraic supergroups over $k$ associated with basic (including exceptional) Lie superalgebras are constructed in analogy to Chevalley's construction of semisimple algebraic groups  (see \cite{FG12} and \cite{G14}); we shall use the same terminologies (such as basic type, exceptional type) to refer to Lie superalgebras and corresponding supergroups. We shall call $G$ simply connected if $\Gev$ is a simply connected algebraic group, taking advantage of \cite[Proposition~35]{Mas12}. Simple-connected supergroups of basic type exist, and we shall assume the exceptional supergroups in this paper to be simply connected.

The assumption on Chevalley bases in \cite[Theorem 2.8]{SW08} is satisfied for all algebraic supergroups of basic type, by the constructions in \cite{FG12, G14}. Hence we have the following.

\begin{prop}  \cite[Theorem 2.8]{SW08} \cite{MS17}
   \label{prop:equiv}
Let $G$ be an algebraic supergroup of basic type. Then there is a natural equivalence of categories between $G$-$\mathfrak{mod}$ and $(\DG,T)$-$\mathfrak{mod}$.
\end{prop}
\noindent If we further assume $G$ is simply connected, then $(\DG,T)$-$\mathfrak{mod}$ in Proposition~\ref{prop:equiv} above can be replaced by  $\DG$-$\mathfrak{mod}$; cf. \cite[II.1.20]{Jan03}.

A supergroup $G$ of basic type can be constructed as a Chevalley supergroup through a Chevalley basis associated with a standard positive root system $\Phi^+$ as described in \cite[\S3.4 and \S3.5]{SW08},  \cite[3.3]{FG12} and \cite[\S3]{G14}. Therefore, we have a standard Borel subgroup $B$ corresponding to $\Phi^+$, which contains a maximal torus $T$. The distribution superalgebra $\DG$ contains $\Dist(B)$ as a
subalgebra. Set $\Lie(B) =\bbb$. Let $X(T)$ be the character group of $T$.
For $ \la  \in X(T)$, we denote the Verma module of $\DG$ by
\[
M(\la) = \DG \otimes_{\Dist(B)} k_\la,
\]
where $k_\la$ is the one-dimensional $\Dist(B)$-module of weight $\la$. 
The $\DG$-module $M(\la)$ has a unique simple quotient $L(\la)$, and furthermore the $\DG$-modules $L(\la)$ are non-isomorphic for distinct $\la \in X(T)$. By definition, $L(\la)$ is $X(T)$-graded and thus a $T$-module.
Denote by $X^+(T)$ the set of $G_\oo$-dominant integral weights (with respect to $\Phi^+$).

\begin{lem}  \cite[Lemma 4.1]{SW08}
 \label{lem:simple}
Every simple module in the category $(\DG,T)$-$\mathfrak{mod}$ is isomorphic to a finite-dimensional highest weight module $L(\la)$ for some $\la \in X^+(T)$, and vice versa.
\end{lem}
By Proposition~\ref{prop:equiv} and Lemma~\ref{lem:simple}, the classification of simple $G$-modules can be reformulated as the determination of the following set:
\begin{equation}
   \label{X+T}
   X^\dagger(T) =\big\{\la \in X^+(T) ~\big |~ L(\la) \text{ is finite dimensional} \big\}.
\end{equation}
For general supergroups of basic type, $X^\dagger(T)$ turns out to be a nontrivial proper subset of $X^+(T)$.

\begin{rem}
 \label{rem:osp}
For a supergroup $G$ of type $\mathfrak{spo}(2n|\ell)$, the subset $X^\dagger(T) \subset X^+(T)$ was determined explicitly in \cite{SW08}. Note the supergroup $G$ therein has even subgroup $\Gev = \text{Sp}_{2n} \times \text{SO}_\ell$ and hence is not simply connected. For a simply connected group of type $\mathfrak{spo}(2n|\ell)$, one would have additional simple modules $L(\la)$, where $\la \in X^+(T)$ is  of the form $\la \in \sum_{i<0} \Z\delta_i + \sum_{j>0} (\hf+\Z)\delta_j$ in the notation of \cite[\S3.3-3.4]{SW08};  This follows from Proposition~\ref{prop:EulerFinite} below.   
\end{rem}

We denote by $L'(\la)$ and $L''(\la)$ the highest weight $\DG$-modules with respect to positive systems $\Phi'^{+}$ and $\Phi''^{+}$, respectively.
 \begin{lem}  \cite[Lemma~4.2]{BKu03} \cite[Lemma 5.7]{SW08}
   \label{lem:oddref}
Let $\la \in X(T)$, and let $\beta$ be an odd isotropic root for $\g$. Suppose that $\Phi'^{+}$ and $\Phi''^{+}$ are two positive systems of $\g$ such that $\Phi''^{+} =\Phi'^{+} \cup \{-\beta\} \backslash \{\beta \}$. Then,
\[
L'' (\la) \cong \left \{
 \begin{array}{ll}
  L'(\la) & \text{ if } (\la, \beta) \equiv 0  \pmod p, \\
  L'(\la -\beta) & \text{ if } (\la, \beta) \not \equiv 0  \pmod p.
 \end{array}   \right.
\]
\end{lem}

We shall say $\Phi''^{+}$ is obtained from $\Phi'^{+}$ by an odd reflection in the setup of Lemma~\ref{lem:oddref}. Often we shall abbreviate $a\equiv b\pmod p$ as $a\equiv b$ later on. In the coming sections dealing with exceptional supergroups, we shall be very explicit about the (positive) root systems and odd reflections.

\begin{lem}
  \label{lem:rational}
Let $L=L(\la)$, for $\la \in X^+(T)$.  Suppose that $L$ is isomorphic to $L^{\bbb'}(\la')$ with $\la' \in X^+(T)$, for every Borel subalgebra $\bbb'$ that is obtained from $\bbb$ by a sequence of odd reflections. Then $L$ is locally finite as a $\DGev$-module, i.e., it is a rational $G_\oo$-module.
\end{lem}

\begin{proof}
We recall the following observation (cf., e.g., \cite{Se11, Ma14}):

{\em For every positive even root $\alpha$ in $\Phi_\oo^+$, either $\alpha/2$ (if it is a root) or $\alpha$ appears as a simple root in some simple system $\Pi'$ associated to $\bbb'$.}


Denote by $SL_{2,\al}$ the root subgroup of $G$ associated to $\alpha$. Then by the assumption of the lemma, $\Dist(SL_{2,\al})$ acts on $L$ locally finitely (i.e., $L$ is a rational $SL_{2,\al}$-module). It follows that $L$ is a rational $G_\oo$-module, or equivalently, $L$ is locally finite as a $\DGev$-module by Proposition~\ref{prop:equiv}.
\end{proof}

\begin{lem}
  \label{lem:lf=fd}
If a finitely generated $\DG$-module $M$ is locally finite as a $\DGev$-module, then $M$ is finite dimensional.
\end{lem}

\begin{proof}
Since $\DG$ is finitely generated over the algebra $\DGev$,  as a $\DGev$-module $M$ is also finitely generated. Together with the locally finiteness assumption, this implies that  $M$ is finite dimensional.
\end{proof}
The combination of Proposition~\ref{prop:equiv}, Lemmas~\ref{lem:oddref}, \ref{lem:rational} and \ref{lem:lf=fd} provides us with an effective approach of classifying simple $G$-modules. Indeed, the problem of determining the finite-dimensional irreducible modules is thus reduced to determining the weights that remain to be $G_{\bar 0}$-dominant integral when transformed to highest weights with respect to any Borel (with fixed even part).

\subsection{Euler characteristic}

Let $H$ be a closed subgroup of an algebraic supergroup $G$
such that the quotient superscheme $G/H$ is locally decomposable (cf. \cite[the paragraph above Lemma~2.1]{B06}) and $G_\oo/H_\oo$ is projective; that is, the superscheme $X=G/H$ satisfies the assumptions (Q5)-(Q6) in \cite[\S 2]{B06}.

We refer to \cite[II.2]{Jan03} and \cite[\S6]{BK03} for the precise definitions for induction and restriction functors below. Below, for a superspace $M$, we shall use $S(M)$ to denote the corresponding supersymmetric algebra.

\begin{lem} (\cite[Corollary 2.8]{B06})
  \label{lem:Euler}
For any finite-dimensional $H$-module $M$, we have
\[
\sum_{i\geq 0} (-1)^i [\res^G_{\Gev} R^i \ind^G_H M] =\sum_{i\geq 0}(-1)^i [R^i\ind^{\Gev}_{\Hev}S\big(\left(\Lie G \slash\Lie H\right)^*_{\bar 1}\big)\otimes M],
\]
where the equality is understood in the Grothendieck group of $\Gev$-modules.
\end{lem}

Now we take $G$ to be an algebraic supergroup of basic type, $H=B^-$ to be the opposite Borel subgroup.
Since $G_\oo/B^-_\oo$ is projective and $G/B^-$ is locally decomposable (cf. \cite{MZ17} and \cite{Z18}), Lemma~\ref{lem:Euler} is applicable.
For $M=k_\lambda$, we define $H^i(\lambda):=R^i \ind^G_{B^-}(k_\lambda)$ and then the Euler characteristic
\[
\chi(\la):= \sum_{i\geq 0} (-1)^i\; \ch \, H^i(\lambda).
\]
By Lemma~\ref{lem:Euler} we have the following formula for the Euler characteristic
\begin{align}
 \label{eq:Euler}
\chi (\la)
 =  \sum_{i\geq 0}(-1)^i\; \ch \, R^i \ind^{\Gev}_{\Bev}S\big(\left(\g \slash \bbb^-\right)^*_{\bar 1}\big)\otimes k_\lambda,
 \end{align}
where $\bbb^-$ is the opposite Borel subalgebra. Since the Euler characteristic is additive on short exact sequences, it suffices to determine the Euler characteristic on the composition factors of the $\Bev$-module $S\big((\g \slash \bbb^-)^*_{\bar 1}\big)\otimes k_\lambda$. Recall that the supersymmetric algebra of a purely odd space is the exterior algebra in the usual sense. Let $W$ be the Weyl group of $\g$. Since $\Pi_{\beta\in \Phi^+_{\bar 1}}(e^{\beta\over 2}+e^{-\beta\over 2})$ is $W$-invariant, it follows by \eqref{eq:Euler} and Lemma~\ref{lem:Euler} that
\begin{align*}
\chi(\la)
 =& {\sum_{w\in W}(-1)^{\ell(w)}w(e^{(\lambda+\rho_{\bar 0})}(\Pi_{\beta\in \Phi^+_{\bar 1}}(1+e^{-\beta}))) \over \Pi_{\alpha\in \Phi^+_{\bar 0}}(e^{\alpha\over 2}-e^{-\alpha\over 2})}
 \cr
 =& {\sum_{w\in W}(-1)^{\ell(w)}w(e^{(\lambda+\rho)}(\Pi_{\beta\in \Phi^+_{\bar 1}}(e^{\beta\over 2}+e^{-\beta\over 2}))) \over \Pi_{\alpha\in \Phi^+_{\bar 0}}(e^{\alpha\over 2}-e^{-\alpha\over 2})}
 \cr
  =&{\Pi_{\beta\in \Phi^+_{\bar 1}}(e^{\beta\over 2}+e^{-\beta\over 2})
 \over \Pi_{\alpha\in \Phi^+_{\bar 0}}(e^{\alpha\over 2}-e^{-\alpha\over 2})}
 \sum_{w\in W}(-1)^{\ell(w)}e^{w(\lambda+\rho)}.
  \end{align*}
Here as usual $\ell(w)$ denotes the length of $w\in W$, and $\rho$ is the Weyl vector given by
\[
\rho=\rho_{\bar 0}-\rho_{\bar 1},
\quad \text{ where }\rho_{\bar 0}=\hf\sum_{\alpha\in\Phi^+_{\bar 0}} \alpha,\quad
\rho_{\bar 1} =\hf\sum_{\beta\in\Phi^+_{\bar 1}}\beta.
\]

\begin{prop}
  \label{prop:EulerWeyl}
Let $\la \in X^+(T)$. The Euler characteristic  is given by
\begin{align*}
\chi (\la)
={\Pi_{\beta\in \Phi^+_{\bar 1}}(e^{\beta\over 2}+e^{-\beta\over 2})
 \over \Pi_{\alpha\in \Phi^+_{\bar 0}}(e^{\alpha\over 2}-e^{-\alpha\over 2})}
 \sum_{w\in W}(-1)^{\ell(w)}e^{w(\lambda+\rho)}.
  \end{align*}
\end{prop}

\begin{prop}
  \label{prop:EulerFinite}
Let $\la \in X^+(T)$ be such that $\la+\rho$ is $G_\oo$-dominant and regular. Then $L(\la)$ is finite dimensional.
\end{prop}

\begin{proof}
By the same arguments as in \cite[Corollary 2.8, Lemma 4.2]{B06}, all $H^i(\lambda)$ are finite-dimensional $G$-modules. By assumption $\lambda+\rho$ is $G_\oo$-dominant and regular, and hence, the highest weight of the Euler characteristic in Proposition~ \ref{prop:EulerWeyl} equals $\lambda+\rho +(\rho_1-\rho_0)=\lambda$.
The proposition now follows from Proposition \ref{prop:equiv} and Lemma \ref{lem:simple}.
\end{proof}

\section{Modular representations of the supergroup of type $\D$}
  \label{sec:D}

\subsection{Weights and roots for $\D$}

The Lie superalgebra $\g=\D$ is a family of simple Lie superalgebras of basic type, which depends on a parameter $\ka\in k\setminus\{0,-1\}$. There are  isomorphisms of Lie superalgebras with different parameters 
\begin{equation}
\label{D:iso}
D(2|1;\ka)  \cong
D(2|1; -1-\ka^{-1})  \cong
D(2|1;\ka^{-1}).
\end{equation}
Then $\g =\g_\oo \oplus \g_\one$, where $\g_\oo \cong \sll \oplus \sll \oplus \sll$ and, as a $\g_{\bar 0}$-module,  $\g_\one \cong k^2\boxtimes k^2 \boxtimes k^2$. Here $k^2$ is the natural  representation of $\sll$.

Let $\h^*$ be the dual of the Cartan subalgebra with 
basis $\{\delta, \ep_1,\ep_2\}$. We equip $\h^*$ with a $k$-valued bilinear form $(\cdot,\cdot)$ such that $\{\delta, \ep_1,\ep_2\}$ are orthogonal and
\begin{align}
 \label{form:D}
(\delta, \delta) = -(1+\ka),
\quad
(\ep_1, \ep_1) = 1,
\quad
(\ep_2, \ep_2) = \ka.
\end{align}
The root system for $\g =\g_\oo \oplus \g_\one$ is denoted by $\Phi =\Phi_{\bar 0} \cup \Phi_{\bar 1}$.
The set of simple roots of the standard simple system in $\h^*$ of $\D$ is chosen to be
\begin{align*}
\Pi =\{\alpha_0=\delta-\ep_1-\ep_2,\alpha_1=2\ep_1,\alpha_2=2\ep_2\}.
\end{align*}
The Dynkin diagram associated to $\Pi$ is depicted as follows:

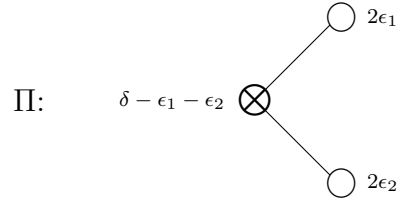
\begin{center}
\hskip-2cm
\begin{tikzpicture}
\node at (0,0) {$\bigotimes$};
\draw (0,0)--(1,1);
\draw (0,0)--(1,-1);
\node at (1.15,1.1) {$\bigcirc$};
\node at (1.15,-1.1) {$\bigcirc$};
\node at (-1.1,0) {\tiny $\delta-\ep_1-\ep_2$};
\node at (1.7,1.1) {\tiny $2\ep_1$};
\node at (1.7,-1.1) {\tiny $2\ep_2$};
\node at (-3,0) { $\Pi$:};
\end{tikzpicture}
\end{center}
The set of positive roots is $\Phi^+ =\Phi^+_\oo \cup \Phi^+_\one$, where
\begin{align*}
\Phi^+_{\bar 0}=\{2\delta,2\ep_1,2\ep_2\},\quad
\Phi^+_{\bar 1}=\{\delta-\ep_1-\ep_2,\delta+\ep_1-\ep_2,\delta-\ep_1+\ep_2,\delta+\ep_1+\ep_2\}.
\end{align*}
One computes the Weyl vector
\begin{align*}
\rho =-\delta+\ep_1+\ep_2 \; (=-\alpha_0).
\end{align*}
Let $$X=\Z\delta+\Z\ep_1+\Z\ep_2$$ denote the weight lattice of $\g$.
%

We denote the positive odd roots by
\begin{equation}  \label{eq:betaD}
\beta_1=\delta-\ep_1-\ep_2, \quad
\beta_2=\delta+\ep_1-\ep_2, \quad
\beta_3=\delta-\ep_1+\ep_2, \quad
\beta_4=\delta+\ep_1+\ep_2.
\end{equation}

There are 4 conjugate classes of positive systems under the Weyl group action. The 4 positive systems containing $\Phi^+_\oo$ admit the following simple systems $\Pi^i$ $(0\le i\le 3)$, which are obtained from one another by applying odd reflections (see \cite[\S1.4]{CW12} for an introduction to odd reflections):
    \begin{align*}
    &\Pi^{0}:=\Pi=\{\delta-\ep_1-\ep_2, \; 2\ep_1, \; 2\ep_2\},
    \cr
   &\Pi^{1}:=r_{\beta_1}(\Pi)=\{-\delta+\ep_1+\ep_2,\;  \delta+\ep_1-\ep_2, \;  \delta-\ep_1+\ep_2\},
   \cr
   &\Pi^{2}:=r_{\beta_2}(\Pi^{1})=\{2\ep_1, \;  - \delta-\ep_1+\ep_2, \;  2\delta \},
   \cr
   &\Pi^{3}:=r_{\beta_3}(\Pi^{1})=\{2\ep_2, \;  2\delta, \;   -\delta+\ep_1-\ep_2\}.
    \end{align*}
The Dynkin diagrams of $\Pi^1$, $\Pi^2$, and $\Pi^3$ are respectively as follows:

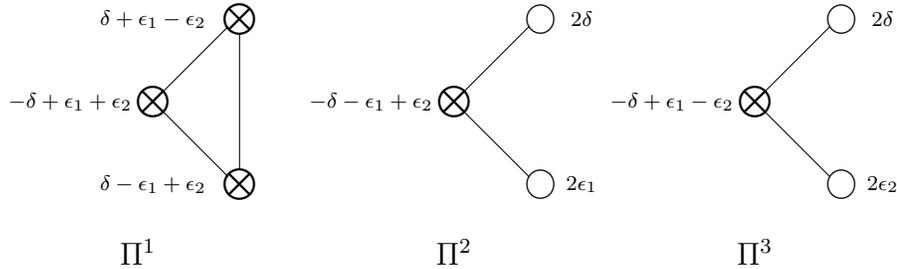
\begin{center}
\hskip-1cm
\begin{tikzpicture}
\node at (-6,0) {$\bigotimes$};
\draw (-6,0)--(-5,1);
\draw (-6,0)--(-5,-1);
\draw (-4.85,-.9)--(-4.85,.9);
\node at (-4.85,1.1) {$\bigotimes$};
\node at (-4.85,-1.1) {$\bigotimes$};
\node at (-7.1,0) {\tiny $-\delta+\ep_1+\ep_2$};
\node at (-6,1.1) {\tiny $\delta+\ep_1-\ep_2$};
\node at (-6,-1.1) {\tiny $\delta-\ep_1+\ep_2$};
\node at (-6.2,-2) { $\Pi^1$};
\node at (-2,0) {$\bigotimes$};
\draw (-2,0)--(-1,1);
\draw (-2,0)--(-1,-1);
\node at (-.85,1.1) {$\bigcirc$};
\node at (-.85,-1.1) {$\bigcirc$};
\node at (-3.1,0) {\tiny $-\delta-\ep_1+\ep_2$};
\node at (-.3,1.1) {\tiny $2\delta$};
\node at (-.3,-1.1) {\tiny $2\ep_1$};
\node at (-2,-2) { $\Pi^2$};
\node at (2,0) {$\bigotimes$};
\draw (2,0)--(3,1);
\draw (2,0)--(3,-1);
\node at (3.15,1.1) {$\bigcirc$};
\node at (3.15,-1.1) {$\bigcirc$};
\node at (.9,0) {\tiny $-\delta+\ep_1-\ep_2$};
\node at (3.7,1.1) {\tiny $2\delta$};
\node at (3.7,-1.1) {\tiny $2\ep_2$};
\node at (2,-2) { $\Pi^3$};
\end{tikzpicture}
\end{center}
The corresponding positive systems are denoted by $\Phi^{i+}$, for $0\le i \le 3$, with  $\Phi^{0+}=\Phi^+$, and the corresponding Borel subalgebras of $\g$ are denoted by $\bbb^i$.

\subsection{Highest weight computations}
  \label{subsec:comp}

The simply connected algebraic supergroup $G$ of type $\D$ was constructed in \cite{G14}.
With respect to the standard Borel subalgebra $\bbb$  (associated to $\Phi^+$), we have
\[
X^+(T) =\{\la =d\delta + a \ep_1 + b \ep_2 \in X~|~a,b,d\in \N\}.
\]
Denote the simple $\DG$-module of highest weight $\la$ by $L(\la)$, where $\la =d\delta + a \ep_1 + b \ep_2 \in X^+(T)$.
We denote by $\la^i$ the highest weight of $L(\la)$ with respect to $\Pi^i$, for $0\le i \le 3$. So $\la^0=\la$.
We shall apply \eqref{form:D} and Lemma~\ref{lem:oddref} repeatedly to compute $\la^{i}$, for $1\le i\le 3$.

By using \eqref{form:D} we have
\[
(\la, \beta_1) = -d(1+\ka) -a -b \ka =-(a+d) -(b+d)\ka.
\]
We now divide into 2 cases (1)-(2).
\begin{enumerate}
\item  
Assume $x_1:=(a+d) +(b+d)\ka  \not\equiv 0 \pmod p$.
Then
\[
\la^{1} =\la -\beta_1 =(d-1)\de +(a+1)\ep_1 +(b+1)\ep_2.
\]

First we compute
$(\la^1, \beta_2) = -(d-1)(1+\ka) +(a+1) -(b+1)\ka =(a-d+2) -(b+d)\ka$, and then further divide into 2 subcases (a)-(b).
  \vspace{2mm}
   \begin{enumerate}
   \item  
   If $y_1:=(a-d+2) -(b+d)\ka \not\equiv 0\pmod p$, then
   \[
   \la^{2} =\la^1 -\beta_2 =(d-2)\de +a \ep_1 + (b+2) \ep_2.
   \]
  \item 
  If $y_1 \equiv 0\pmod p$, then
 $ \la^{2} =\la^1 =(d-1)\de +(a+1)\ep_1 +(b+1)\ep_2.$
  \end{enumerate}

  \vspace{3mm}
    We  also have $(\la^1, \beta_3) = -(d-1)(1+\ka) -(a+1) +(b+1)\ka =-(a+d) +(b-d+2)\ka$, and then divide into 2 subcases (a$'$)-(b$'$).
     \begin{enumerate}
       \item[(a$'$)]  
        If $z_1:=-(a+d) +(b-d+2)\ka \not\equiv 0\pmod p$, then
        \[
          \la^{3} =\la^1 -\beta_3 =(d-2)\de +(a+2) \ep_1 +b\ep_2.
        \]
      \item[(b$'$)]  
       If $z_1 \equiv 0\pmod p$, then
        $\la^{3} =\la^{1} =(d-1)\de +(a+1)\ep_1 +(b+1)\ep_2.$
       \end{enumerate}

  \vspace{3mm}

 \item  
Assume $x_1\equiv 0 \pmod p$.
Then $\la^{1} =\la =d\delta + a \ep_1 + b \ep_2.$

First we compute
$(\la^1, \beta_2) = -d(1+\ka) +a -b\ka =(a-d) -(b+d)\ka$, and then divide into 2 subcases (a)-(b).
  \vspace{2mm}
   \begin{enumerate}
   \item  
   If $y_2:=(a-d) -(b+d)\ka \not\equiv 0\pmod p$, then
   \[
   \la^{2} =\la^1 -\beta_2 =(d-1)\de +(a-1) \ep_1 + (b+1) \ep_2.
   \]
  \item 
  If $y_2 \equiv 0\pmod p$, then
 $ \la^{2} =\la^1 =d\delta + a \ep_1 + b \ep_2.$
  \end{enumerate}

  \vspace{3mm}
    We  also have $(\la^1, \beta_3) = -d(1+\ka) -a +b\ka =-(a+d) +(b-d)\ka$, and then divide into 2 subcases (a$'$)-(b$'$).
     \begin{enumerate}
       \item[(a$'$)]  
        If $z_2:=-(a+d) +(b-d)\ka \not\equiv 0\pmod p$, then
        \[
          \la^{3} =\la^1 -\beta_3 =(d-1)\de +(a+1) \ep_1 +(b-1)\ep_2.
        \]
      \item[(b$'$)]  
       If $z_2 \equiv 0\pmod p$, then
        $\la^{3} =\la^{1} =d\de +a\ep_1 +b\ep_2.$
       \end{enumerate}
\end{enumerate}

\subsection{Simple modules for the supergroup $\D$}

\begin{thm}
 \label{thm:D}
 Let $p>2$. Let $G$ be the supergroup of type $\D$.
 A complete list of inequivalent simple $G$-modules consists of $L(\la)$, where
$\la =d\delta + a\ep_1 + b\ep_2$, with $d,a,b \in \N$, such that
one of the following conditions is satisfied:
 \begin{enumerate}
        \item $d=0$, and $a\equiv b\equiv 0 \pmod p$;
        \item  $d =1$, and $(a+1) -(b+1)\ka \equiv 0 \pmod p$;
        \item  $d =1$, and $(a+1) +(b+1)\ka \equiv 0 \pmod p$;
        \item $d\ge 2$, (and $a,b \in \N$ are arbitrary).
\end{enumerate}
\end{thm}

\begin{proof}
From the computations in \S\ref{subsec:comp} on the highest weights $\la^i$ $(1\le i \le 3)$ and their associated conditions, we obtain the following (mutually exclusive) sufficient and necessary conditions for $L(\la)$ to be finite dimensional:
       \begin{enumerate}
        \item[(i)] $d=0$, $(a+d) +(b+d)\ka \equiv 0$, $(a-d) -(b+d)\ka \equiv 0$, $-(a+d) +(b-d)\ka \equiv 0$;
        \item[(ii)]  $d =1$,  $(a+d) +(b+d)\ka  \not\equiv 0$, $(a-d+2) -(b+d)\ka \equiv 0$, $-(a+d) +(b-d+2)\ka \equiv 0$;
       \item[(iii-a)]
         $d =1$,  $(a+d) +(b+d)\ka \equiv 0$, $(a-d) -(b+d)\ka \not\equiv 0$ with $a \ge 1$, $-(a+d) +(b-d)\ka \not\equiv 0$ with $b \ge 1$;
         \item[(iii-b)]  $d =1$,  $(a+d) +(b+d)\ka \equiv 0$, $(a-d) -(b+d)\ka \equiv 0$, $-(a+d) +(b-d)\ka \not\equiv 0$ with $b \ge 1$;
        \item[(iv-a)]
         $d =1$,  $(a+d) +(b+d)\ka \equiv 0$, $(a-d) -(b+d)\ka \not\equiv 0$ with $a \ge 1$, $-(a+d) +(b-d)\ka \equiv 0$;
         \item[(iv-b)]  $d =1$,  $(a+d) +(b+d)\ka \equiv 0$, $(a-d) -(b+d)\ka \equiv 0$, $-(a+d) +(b-d)\ka \equiv 0$;
        \item[(v)] $d\ge 2$, (and $a,b \in \N$ are arbitrary).
        \end{enumerate}

In Case (i), we obtain $d=0$ and $a\equiv b \equiv 0$, that is, Condition (1) in the theorem. Case~ (v) is the same as Condition (4).

Condition~ (ii) with the help of Condition~ (iii-a) simplifies to Condition (2).

We note that the seemingly additional constraints $a\ge 1$ and $b\ge 1$ in (iii-a)-(iii-b) as well as (iv-a)-(iv-b) follow automatically from the other conditions. Therefore, Conditions (iii-a)-(iv-b) simplify to Condition~(3).

 The theorem is proved.
\end{proof}

\begin{rem}
Theorem~\ref{thm:D} makes sense over $\C$, providing an odd reflection approach to the classification of finite-dimensional simple modules over $\C$ (due to \cite{Kac77}).  Indeed this classification can be read off from Theorem~\ref{thm:D} by regarding $p=\infty$.
\end{rem}

\section{Modular representations of the supergroup of type $G(3)$}
  \label{sec:G}

 \subsection{Weights and roots for the supergroup $G(3)$}
Let $\g=\g_\oo\oplus\g_\one$ be the exceptional simple Lie superalgebra $G(3)$.
 We assume $\ep_1, \ep_2, \ep_3$ satisfy the linear relation
 \[
 \ep_1 +\ep_2 +\ep_3 =0.
 \]
 The root system is $\Phi =\Phi_\oo \cup \Phi_\one$.
 We choose the standard simple system $\Pi =
\{\alpha_1, \alpha_2, \alpha_3\}$, where
\[
\alpha_1 = \ep_2 -\ep_1,
\quad
\alpha_2=\ep_1,
\quad
\alpha_3 =  \delta +\ep_3.
\]

The Dynkin diagram associated to $\Pi$ is depicted as follows:
\vskip .5cm
\begin{center}
\begin{tikzpicture}
\node at (0,0) {$\bigcirc$};
\draw (0.2,0)--(1.155,0);
\draw (0.15,0.1)--(1.2,0.1);
\draw (0.15,-0.1)--(1.2,-0.1);
\node at (1.35,0) {$\bigcirc$};
\node at (0.75,0) {\Large $>$};
\draw (1.52,0)--(2.52,0);
\node at (2.7,0) {$\bigotimes$};
\node at (-0.3,-.5) {\tiny $\al_1=\ep_2-\ep_1$};
\node at (1.3,-.5) {\tiny $\al_2=\ep_1$};
\node at (2.9,-.5) {\tiny $\al_3=\delta+\ep_3$};
\node at (-2,0) { $\Pi$:};
\end{tikzpicture}
\end{center}

Then the standard positive roots are $\Phi^+ =\Phi_\oo^+ \cup \Phi_\one^+$,
where
\[
\Phi_\oo^+ =\{2\delta, \ep_1, \ep_2, -\ep_3, \ep_2 -\ep_1, \ep_1 -\ep_3, \ep_2 -\ep_3 \},
\qquad
\Phi_\one^+ =\{\delta, \,   \delta \pm \ep_i\mid 1\le i \le 3\}.
\]
The Weyl vector for $\g$ is
\begin{equation}
  \label{rho:ep}
 \rho =- \frac52\, \delta +2\ep_1 +3\ep_2, \qquad \rho_{\bar 1} =\frac{7}{2}\delta.
\end{equation}

We have $\g_\oo \cong \ G_2 \oplus \sll$ and $\g_\one \cong k^7 \boxtimes k^2$ as an adjoint $\g_\oo$-module,
where $k^7$ denotes denotes the $7$-dimensional simple $G_2$-module and, as before, $k^2$ the natural $\sll$-module. Note that $\{\al_1, \al_2\}$ forms a simple system of $G_2$, and
we denote by $\om_1, \om_2$ the corresponding fundamental weights of $G_2$.
We have
\begin{align*}
\om_1 =\ep_1 + 2\ep_2, \qquad
\om_2 &=\ep_1 +\ep_2;
\\
\ep_1= 2\om_2 -\om_1,  \qquad
\ep_2 &= \om_1 -\om_2.
\end{align*}
We can rewrite the formulae for $\rho$ in \eqref{rho} as
\begin{equation}
  \label{rho}
\rho =-\frac52 \de +\om_1 +\om_2.
\end{equation}

Denote the weight lattice of $\g$ by
\[
X =\Z\delta \oplus X_2,
\]
where
\[
X_2=\Z\om_1 \oplus \Z\om_2 =\Z\ep_1 \oplus \Z\ep_2
\]
is the weight lattice of $G_2$.

The bilinear form $(\cdot, \cdot)$ on $X$ is given by
\[
(\delta, \delta) =-2, \quad (\delta, \ep_i) =0,
\quad
(\ep_i, \ep_i) =2, \quad (\ep_i, \ep_j) =-1, \quad \text{ for } 1\le i \neq j \le 3.
\]
It follows that
\begin{align}  \label{eq:omep}
\begin{split}
(\om_1, \ep_1) =0,  \quad
(\om_1, \ep_2) =3, \quad (\om_1, \ep_3) =-3,
\\
(\om_2, \ep_1) =1, \quad (\om_2, \ep_2) =1,
\quad (\om_2, \ep_3) =-2.
\end{split}
\end{align}

%

Denote the following positive odd roots of $G(3)$ by
\begin{equation}  \label{eq:beta}
\beta_1=\delta+\ve_3, \quad \beta_2=\delta-\ve_2, \quad \beta_3=\delta-\ve_1.
\end{equation}

There are 4 conjugate classes of positive systems under the Weyl group action. The 4 positive systems containing $\Phi^+_\oo$ admit the following simple systems $\Pi^i$ $(0\le i\le 3)$, which are obtained from one another by applying odd reflections (cf. \cite[\S1.4]{CW12}):
    \begin{align*}
    &\Pi^{0}:=\Pi=\{\varepsilon_2-\varepsilon_1, \varepsilon_1,  \delta+\varepsilon_3\},
    \cr
   &\Pi^{1}:=r_{\beta_1}(\Pi)=\{\ve_2-\ve_1, \delta-\ve_2,  -\delta-\ve_3 \},
   \cr
   &\Pi^{2}:=r_{\beta_2}(\Pi^{1})=\{\delta-\ve_1, -\delta+\ve_2,  \ve_1 \},
   \cr
   &\Pi^{3}:=r_{\beta_3}(\Pi^{2})=\{-\delta+\ve_1, \ve_2-\ve_1, \delta\}.
    \end{align*}
The Dynkin diagrams of $\Pi^1$, $\Pi^2$, and $\Pi^3$ are respectively as follows:
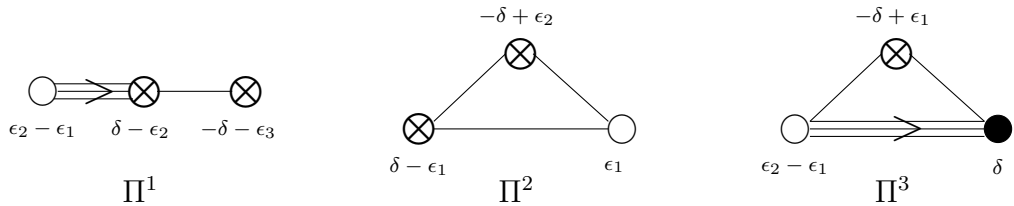
\begin{center}
\begin{tikzpicture}
\node at (0,0.5) {$\bigcirc$};
\draw (0.2,0.5)--(1.155,0.5);
\draw (0.15,0.6)--(1.2,0.6);
\draw (0.15,0.4)--(1.2,0.4);
\node at (1.35,0.5) {$\bigotimes$};
\node at (0.75,0.5) {\Large $>$};
\draw (1.52,0.5)--(2.52,0.5);
\node at (2.7,0.5) {$\bigotimes$};
\node at (0,0) {\tiny $\ep_2-\ep_1$};
\node at (1.3,0) {\tiny $\delta-\ep_2$};
\node at (2.6,0) {\tiny $-\delta-\ep_3$};
\node at (1.3,-.8) { $\Pi^1$};
\node at (5,0) {$\bigotimes$};
\draw (5.2,.1)--(6.155,1);
\node at (6.35,1) {$\bigotimes$};
\draw (6.52,1)--(7.52,.1);
\draw (5.2,0)--(7.52,0);
\node at (7.7,0) {$\bigcirc$};
\node at (5,-.5) {\tiny $\delta-\ep_1$};
\node at (6.3,1.5) {\tiny $-\delta+\ep_2$};
\node at (7.6,-.5) {\tiny $\ep_1$};
\node at (6.3,-.8) { $\Pi^2$};
\node at (10,0) {$\bigcirc$};
\draw (10.2,.1)--(11.155,1);
\node at (11.35,1) {$\bigotimes$};
\draw (11.52,1)--(12.52,.1);
\draw (10.2,0)--(12.52,0);
\draw (10.2,-.1)--(12.52,-.1);
\draw (10.2,.1)--(12.52,.1);
\node at (11.5,0) {\Large $>$};
\draw[fill] (12.7,0) circle (1ex);
\node at (10,-.5) {\tiny $\ep_2-\ep_1$};
\node at (11.3,1.5) {\tiny $-\delta+\ep_1$};
\node at (12.7,-.5) {\tiny $\delta$};
\node at (11.3,-.8) { $\Pi^3$};
\end{tikzpicture}
\end{center}

The corresponding positive systems are denoted by $\Phi^{i+}$, for $0\le i \le 3$, with  $\Phi^{0+}=\Phi^+$, and the corresponding Borel subalgebras of $\g$ are denoted by $\bbb^i$.



%
%
\subsection{Highest weight computations}
  \label{subsec:comp}

The (simply connected) algebraic supergroup $G$ of type $G(3)$ was constructed in \cite{FG12}. With respect to the standard Bore subalgebra $\bbb$ (associated to $\Phi^+$), we have
\begin{equation*}
X^+(T) =\{ \la =n\delta + r \om_1 +s \omega_2 \in X \mid n,r,s\in \N
\}.
\end{equation*}
 Denote by $L(\la) =L^\bbb(\la)$ the irreducible $\DG$-module  of highest weight $\la$ with respect to the standard Borel subalgebra $\bbb$, where
\[
\la =d\delta + r \om_1 +s \omega_2 \in X^+(T).
\]
Assume the simple module $L(\la) =L^\bbb(\la)$ has $\bbb^{i}$-highest weight $\la^{i}$, for $i=1,2,3$.
We shall apply \eqref{eq:omep} and Lemma~\ref{lem:oddref} repeatedly to compute $\la^{i}$, for $1\le i\le 3$.
We have
$(\la, \beta_1) =-2d-3r-2s.$ We now divide into 2 cases (1)--(2).
\begin{enumerate}
\item  
Assume $\it{x_1:=2d+3r+2s}$ $\not\equiv 0 \pmod p$.
Then
\[
\la^{1} =\la -\beta_1 =(d-1)\de +r\om_1 +(s+1)\om_2.
\]
We obtain
$(\la^1, \beta_2) =-2d-3r-s+1.$ We then divide into 2 subcases (a)-(b).
  \vspace{2mm}
   \begin{enumerate}
   \item  
  Assume $\it y_1:=2d+3r+s-1$ $\not\equiv 0\pmod p$. Then
   \[
   \la^{2} =\la^1 -\beta_2 =(d-2)\de +(r+1)\om_1 +s\om_2.
   \]
  We  have $(\la^2, \beta_3) =-2d-s+4.$
     \begin{enumerate}
       \item  
        If $\it z_1:=2d+s-4$ $\not\equiv 0\pmod p$, then
        \[
          \la^{3} =\la^2 -\beta_3 =(d-3)\de +r\om_1 +(s+2)\om_2.
        \]
      \item 
       If $z_1 \equiv 0\pmod p$, then
        $\la^{3} =\la^2 =(d-2)\de +(r+1)\om_1 +s\om_2.$
       \end{enumerate}
  \vspace{2mm}
  \item 
  Assume $y_1=2d+3r+s-1 \equiv 0\pmod p$. Then
  \[
   \la^{2} =\la^1 =(d-1)\de +r\om_1 +(s+1)\om_2.
  \]
  We have $(\la^2, \beta_3) =-2d-s+1.$  
       \begin{enumerate}
     \item 
       If $\it z_2:=2d+s-1$ $\not\equiv 0\pmod p$, then
       \[
         \la^{3} =\la^2 -\beta_3 =(d-2)\de +(r-1)\om_1 +(s+3)\om_2.
       \]
    \item 
       If $z_2 \equiv 0\pmod p$, then
        $\la^{3} =\la^2 =\la^1 =(d-1)\de +r\om_1 +(s+1)\om_2.$
      \end{enumerate} 
  \end{enumerate} 
  \vspace{3mm}
\item
Assume $x_1=2d+3r+2s  \equiv 0\pmod p$.
Then
$\la^{1} =\la =d\de +r\om_1 +s\om_2.$
We have
$(\la^1, \beta_2) =-2d-3r-s.$  We then divide into 2 subcases (a)-(b).
  \vspace{2mm}

  \begin{enumerate}
   \item   
   Assume $\it y_2:=2d+3r+s$ $\not\equiv 0\pmod p$. Then
   \[
   \la^{2} =\la^1 -\beta_2 =\la -\beta_2 =(d-1)\de +(r+1)\om_1 +(s-1)\om_2.
   \]
   We have $(\la^2, \beta_3) =-2d-s+3.$
      \begin{enumerate}
      \item  
        If $\it z_3:=2d+s-3$ $\not\equiv 0\pmod p$, then
        \[
        \la^{3} =\la^2 -\beta_3 =(d-2)\de +r\om_1 +(s+1)\om_2.
        \]%
      \item  
         If $z_3 \equiv 0\pmod p$, then
        $\la^{3} =\la^2 =(d-1)\de +(r+1)\om_1 +(s-1)\om_2.$
      \end{enumerate}
  \vspace{2mm}
  \item 
    Assume $y_2=2d+3r+s \equiv 0\pmod p$. Then
    \[
    \la^{2} =\la^1 =\la  =d\de +r\om_1 +s\om_2.
    \]
   We have $(\la^2, \beta_3) =-2d-s.$
      \begin{enumerate}
      \item  
        If $\it z_4:=2d+s$ $\not\equiv 0\pmod p$, then
        \[
        \la^{3} =\la^2 -\beta_3 =(d-1)\de +(r-1)\om_1 +(s+2)\om_2.
        \]%
      \item  
         If $z_4 \equiv 0\pmod p$, then
        $\la^{3} =\la^2 =\la =d\de +r\om_1 +s\om_2.$
      \end{enumerate}
  \end{enumerate}
\end{enumerate}


\begin{prop}
 \label{prop:nec}
Assume $\la =d\delta + r \om_1 +s \omega_2$, for $d,r,s\in \N$.
Then $L(\la)$ is finite dimensional if only if one of the following conditions holds:
\begin{enumerate}
\item 
   \begin{enumerate}
   \item  
       \begin{enumerate}
        \item $d\ge 3$, $2d+3r+2s \not\equiv 0$, $2d+3r+s-1\not\equiv 0$, $2d+s-4\not\equiv 0$;
        \item  $d\ge 2$, $2d+s-4 \equiv 0$, ${3}(r+1) \not \equiv 0$, $2d+3r+2s\not\equiv 0$;
        \end{enumerate}
    \item  
       \begin{enumerate}
        \item  $d\ge 2$, ${3}r\not\equiv 0$, $2d+3r+s-1 \equiv 0$, $2d+3r+2s \not\equiv 0$;
        \item  $d\ge 2$, ${3}r \equiv 0$, $2d+s-1\equiv 0$, $2d+3r+2s\not\equiv 0$;
       \end{enumerate}
  \end{enumerate}
\item 
  \begin{enumerate}
   \item
       \begin{enumerate}
        \item $d\ge 2$, $s\not\equiv 0$, $2d+3r+2s\equiv 0$, $3r+s+3\not\equiv 0$;
        \item $d\ge 1$, $s\not\equiv 0$, $2d+3r+2s\equiv 0$, $3r+s+3\equiv 0$;
        \end{enumerate}
    \item
       \begin{enumerate}
        \item  $s\equiv 0$,  ${2}d\not\equiv 0$, $2d+3r \equiv 0$;
        \item  ${2}d\equiv {3}r\equiv s \equiv 0$.
       \end{enumerate}
  \end{enumerate}
\end{enumerate}
\end{prop}

\begin{proof}
The conditions in the proposition are summary of the dominant conditions for the new highest weights after odd reflections, which were computed in \S\ref{subsec:comp}. We remark that the natural condition on $d$ from the summary in \S\ref{subsec:comp} for the case (1b)(ii) is ``$d\ge 1$", but ``$d=1$" is quickly ruled out by the other conditions $3r\equiv 0, 2d+s-1\equiv 0, 2d+3r+2s\not\equiv 0$.

It follows by Lemmas~\ref{lem:rational} and \ref{lem:lf=fd} that these conditions are also sufficient for $L(\la)$ to be finite dimensional.
\end{proof}

Note the conditions in Proposition~\ref{prop:nec} are obtained without using any division on the conditions arising from odd reflections; some scalars $2, 3$  therein appear to be unnecessary for $p>3$, and they are kept for the case when $p=3$ below.

\subsection{Simple modules for the supergroup $G(3)$ for $p>3$}

We assume the characteristic of the ground field $k$ is $p>3$ in this subsection.
We shall reformulate the conditions in Proposition~\ref{prop:nec} in a more useful form.
We first analyze the case when $d\ge 3$.
\begin{prop}
  \label{prop:d3}
For $\la =d \delta + r \om_1 +s \omega_2 \in X^+(T)$ with $d\ge 3$, the module $L(\la)$ is always finite dimensional.
\end{prop}

\begin{proof}
Recall $\rho$ from \eqref{rho}. The proposition now follows by Proposition~\ref{prop:EulerFinite} since $\la+\rho =(d-\frac52) \de +(r+1)\om_1 + (s+1)\om_2$ is in $X^+(T)$. (Alternatively, the proposition also follows from the analysis in \S\ref{subsec:comp}.)
\end{proof}

We then analyze the case when $d=2$.
\begin{prop}
  \label{prop:d=2}
Let $p>3$. The module $L(\la)$ is finite dimensional, for $\la =2 \delta + r \om_1 +s \omega_2 \in X^+(T)$, if and only if one of the following 3 conditions are satisfied:
\begin{enumerate}
\item[(i)]  $s\equiv 0 \pmod p$;
\item[(ii)]  $3r+ s +3 \equiv 0 \pmod p$;
\item[(iii)]  $3r+2s+4\equiv 0 \pmod p$.
\end{enumerate}
\end{prop}
The three conditions (i)-(iii) in Proposition~\ref{prop:d=2} are not mutually exclusive. Three mutually exclusive conditions are given in \eqref{eq:1aii+2bi}-\eqref{eq:2ai+2aii} below.

\begin{proof}
Let us set $d=2$ in Proposition~\ref{prop:nec}.

Condition~(1a)(ii) becomes $s\equiv 0, r+1\not\equiv 0, 3r+4 \not\equiv 0$, while Condition~(2b)(i) becomes $s\equiv 0, 3r+4\equiv 0$ (and it follows that $r+1\not\equiv 0$). Hence the combination of Conditions~(1a)(ii) and (2b)(i) gives us the following conditions:
\begin{equation}   \label{eq:1aii+2bi}
s\equiv 0, \qquad   r+1\not\equiv 0.
\end{equation}

Condition~(1b)(i) becomes $r\not\equiv 0, 3r+s+3 \equiv 0, 3r+2s+4 \not\equiv 0$, while Condition~(1b)(ii) becomes $r \equiv 0, 3r+s+3\equiv 0, 3r+2s+4\not\equiv 0$.  Hence the combination of Conditions~(1b)(i)-(ii) gives us the following conditions:
\begin{equation}   \label{eq:1bi+1bii} 
3r+s+3\equiv 0, \qquad  3r+2s+4\not\equiv 0.
\end{equation}

Condition~(2a)(i) becomes $s\not\equiv 0, 3r+2s+4 \equiv 0, 3r+s+3\not\equiv 0$, while Condition~(2a)(ii) becomes $s\not\equiv 0, 3r+2s+4\equiv 0, 3r+s+3\equiv 0$. Hence the combination of Conditions~(2a)(i)-(ii) gives us the following conditions:
\begin{equation}   \label{eq:2ai+2aii} 
3r+2s+4\equiv 0, \qquad   s\not\equiv 0.
\end{equation}
So by Proposition~\ref{prop:nec}, $L(\la)$ is finite dimensional, for $\la =2 \delta + r \om_1 +s \omega_2 \in X^+(T)$, if and only if one of the 3 (mutually exclusive) conditions \eqref{eq:1aii+2bi}, \eqref{eq:1bi+1bii}, \eqref{eq:2ai+2aii} holds.

Let us show that Conditions \eqref{eq:1aii+2bi}-\eqref{eq:2ai+2aii} are equivalent to Conditions (i)-(iii) in the proposition. Clearly if $r,s$ satisfy one of Conditions \eqref{eq:1aii+2bi}-\eqref{eq:2ai+2aii}, then they satisfy one of Conditions (i)-(iii).
On the other hand, if $r,s$ satisfy Condition~(i) but not \eqref{eq:1aii+2bi}, that is, $s\equiv   r+1\equiv 0$, then \eqref{eq:1bi+1bii} is satisfied. If $r,s$ satisfy Condition~(ii) but not \eqref{eq:1bi+1bii}, that is, $3r+s+3\equiv  3r+2s+4\equiv 0$, then \eqref{eq:2ai+2aii} is satisfied. Finally, if $r,s$ satisfy Condition~(iii) but not \eqref{eq:2ai+2aii}, that is, $s\equiv  3r+2s+4\equiv 0$, then \eqref{eq:1aii+2bi} is satisfied.

The proof of Proposition~\ref{prop:d=2} is completed.
\end{proof}

We finally analyze the case when $d=1$.
\begin{prop}
  \label{prop:d=1}
Let $p>3$. The module $L(\la)$ is finite dimensional, for $\la = \delta + r \om_1 +s \omega_2 \in X^+(T)$, if and only if one of the following 2 conditions are satisfied:
\begin{enumerate}
\item[(i)]  $s-1\equiv  3r+4 \equiv 0 \pmod p$;
\item[(ii)]  $s \equiv 3r+2 \equiv 0 \pmod p$.
\end{enumerate}
\end{prop}

\begin{proof}
Let us set $d=1$ in Proposition~\ref{prop:nec}. The case $d=1$ only occurs in Cases (2a)(ii) and (2b)(i).
Condition~(2a)(ii) reads $s\not\equiv 0, 3r+2s+2\equiv 0, 3r+s+3\equiv 0$, which is clearly equivalent to (i) in the proposition.
Condition~ (2b)(i) is the same as (ii) above.
\end{proof}

Summarizing Propositions~\ref{prop:nec}, \ref{prop:d3}, \ref{prop:d=2} and \ref{prop:d=1} (and recalling Proposition~\ref{prop:equiv}, Lemma~\ref{lem:simple}), we have established the following.

\begin{thm}  \label{thm:G3}
Let $p>3$. Let $G$ be the supergroup of type $G(3)$.
 A complete list of inequivalent simple $G$-modules consists of $L(\la)$, where $\la = d \delta + r \om_1 +s \omega_2$, with $d,r,s \in \N$, such that one of the following conditions is satisfied:
\begin{enumerate}
\item $d=0$, and ${3} r\equiv s\equiv 0 \pmod p$.

\item
$d=1$, and $r,s$ satisfy either of (i)-(ii) below:
\begin{enumerate}
\item[(i)] $s-1\equiv  3r+4 \equiv 0 \pmod p$;
\item[(ii)] $s \equiv 3r+2 \equiv 0 \pmod p$.
\end{enumerate}

\item
$d=2$, and $r,s$ satisfy either of (i)-(iii) below:
\begin{enumerate}
\item[(i)] $s\equiv 0 \pmod p$;
\item[(ii)] $3r+ s +3 \equiv 0 \pmod p$;
\item[(iii)] $3r+2s+4\equiv 0 \pmod p$.
\end{enumerate}

\item $d\ge 3$, (and $r,s\in \N$ are arbitrary).
\end{enumerate}
\end{thm}

\begin{rem}
Theorem~\ref{thm:G3} makes sense over $\C$, providing an odd reflection approach to the classification of finite-dimensional simple modules over $\C$ (due to \cite{Kac77}; also cf. \cite{Ma14}).  Indeed this classification can be read off from Theorem~\ref{thm:G3} (by regarding $p=\infty$) as follows.
{\it The $\g$-modules $L(\la)$ over $\C$ is finite dimensional if and only $\la = d \delta + r \om_1 +s \omega_2$, for  $d, r, s \in \N$, satisfies one of the 3 conditions:
$
(1) \; d=r=s=0; \;
(2) \; d=2, s=0; \;
(3) \; d \ge 3. $
}
\end{rem}

\subsection{Simple modules for the supergroup $G(3)$ for $p=3$}

  \label{sec:p=3}

The assumption $p>3$ is not really necessary for the definition of $G$ and classification of simple $G$-modules. The (less polished) conditions in Proposition~\ref{prop:nec} remain valid for $p=3$. When one works it through, it turns out to be the same as setting $p=3$ in Theorem~\ref{thm:G3}; note the scalar $3$ in (1) therein. We summarize this in the following.

\begin{thm}
  \label{thm:p=3}
Let $p=3$. Let $G$ be the supergroup of type $G(3)$.
 A complete list of inequivalent simple $G$-modules consists of $L(\la)$, where $\la = d \delta + r \om_1 +s \omega_2$, with $d,r,s \in \N$, such that one of the following conditions is satisfied:
\begin{enumerate}
\item $d=0$, and $s\equiv 0 \pmod 3$;

\item
$d=2$, $s\equiv 0$ or $1 \pmod 3$;

\item
$d\ge 3$, (and $r,s\in \N$ are arbitrary).
\end{enumerate}
\end{thm}

\section{Modular representations of the supergroup of type $F(3|1)$}
  \label{sec:F}

We assume the characteristic of the ground field $k$ is $p>3$ in this section.

\subsection{Weights and roots for $F(3|1)$}

Let $\g=\g_\oo\oplus\g_\one$ be the exceptional simple Lie superalgebra $F(3|1)$ (which is sometimes denoted by $F(4)$ in the literature). We have $\g_\oo \cong \sll\oplus\mathfrak{so}_7 $ and $\g_\one \cong k^2 \boxtimes k^8$ as $\g_\oo$-module,
where $k^8$ here is the $8$-dimensional spin representation of $\mathfrak{so}_7$. The root system of $\g$ can be described via the basis $\{\epsilon_1,\epsilon_2,\epsilon_3,\delta\}$ in $\h^*\cong \bbc^4$ with a non-degenerate bilinear form $(\cdot, \cdot)$ as follows:
\begin{align}\label{bilinear form}
(\delta,\delta)=-3, (\delta,\epsilon_i)=0, (\epsilon_i,\epsilon_i)=1,  (\epsilon_i,\epsilon_j)=0,\quad  i,j=1,2,3, i\ne j.
\end{align}
 The root system $\Phi=\Phi_\oo\cup \Phi_\one$ is as below:
$$\Phi_\oo=\{\pm\delta;\pm\epsilon_i\pm\epsilon_j;\pm\epsilon_i\mid i,j=1,2,3,i\ne j \};\;\;
\Phi_\one=\{{1\over 2}(\pm\delta\pm \epsilon_1\pm\epsilon_2\pm\epsilon_3) \}$$
The standard Borel subalgebra $\bbb$ corresponds to the simple root system
$$\Pi=\big\{\alpha_1:=\epsilon_1-\epsilon_2, \; \alpha_2:=\epsilon_2-\epsilon_2,\; \alpha_3:=\epsilon_3, \;\alpha_4:={1\over 2}(\delta-\epsilon_1-\epsilon_2-\epsilon_3) \big\}.$$
The fundamental weights of $\g_\oo$ associated with the $\g_\oo$-simple roots $\alpha_1,\alpha_2,\alpha_3,\delta$ are:
$$\omega_1 :=\epsilon_1,\quad
 \omega_2 :=\epsilon_1+\epsilon_2,\quad
\omega_3 :={1\over 2}(\epsilon_1+\epsilon_2+\epsilon_3),\quad
\omega_4 :={1\over 2}\delta.$$
Denote the weight lattice by
\[
X =\{\la =a\omega_1 +b\omega_2 +c\omega_3 +d\omega_4 ~|~a,b,c,d\in \Z\}.
\]
Sometimes we simply denote $\la=a\omega_1 +b\omega_2 +c\omega_3 +d\omega_4 \in X$  as
\[
\la=(a,b,c,d)\in \Z^4.
\]
With respect to $\bbb$, the Weyl vector $\rho$ can be expressed in terms of the fundamental weights as
\begin{align*}
\rho=\omega_1+\omega_2+\omega_3-3\omega_4.
\end{align*}

The Dynkin diagram associated to $\Pi$ is depicted as follows:
\vspace{2mm}
\begin{center}
\begin{tikzpicture}
\node at (0,0.5) {$\bigcirc$};
\draw (0.16,0.5)--(1.15,0.5);
\node at (1.35,0.5) {$\bigcirc$};
\node at (2,0.5) {\Large $>$};
\draw (1.5,0.55)--(2.5,0.55);
\draw (1.5,0.45)--(2.5,0.45);
\node at (2.7,0.5) {$\bigcirc$};
\draw (2.9,0.5)--(3.85,0.5);
\node at (4.05,0.5) {$\bigotimes$};
\node at (0,0) {\tiny $\ep_1-\ep_2$};
\node at (1.3,0) {\tiny $\ep_2-\ep_3$};
\node at (2.7,0) {\tiny $\ep_3$};
\node at (4.6,0) {\tiny $\hf(\delta-\ep_1-\ep_2-\ep_3)$};
\node at (-1.8,0.5) { $\Pi$:};
\end{tikzpicture}
\end{center}

From (\ref{bilinear form}), we have
\begin{align}
\label{bracet for fundmental weight}
 \begin{split}
&(\omega_1,\epsilon_1)=1, \; (\omega_1,\epsilon_2)=0,\;(\omega_1,\epsilon_3)=0,\; (\omega_1,\delta)=0,\cr
&(\omega_2,\epsilon_1)=1, \; (\omega_2,\epsilon_2)=1,\;(\omega_2,\epsilon_3)=0,\; (\omega_2,\delta)=0,\cr
&(\omega_3,\epsilon_1)={1\over 2}, \; (\omega_3,\epsilon_2)={1\over 2},\;(\omega_3,\epsilon_3)={1\over 2},\; (\omega_3,\delta)=0,\cr
&(\omega_4,\epsilon_1)=0, \; (\omega_4,\epsilon_2)=0,\;(\omega_4,\epsilon_3)=0,\; (\omega_4,\delta)=-{3\over 2}.
 \end{split}
\end{align}


Denote the positive odd roots for $F(3|1)$ by
\begin{align*}
\ga_1 &={1\over 2}(\delta-\epsilon_1-\epsilon_2-\epsilon_3), \quad
\ga_2={1\over 2}(\delta-\epsilon_1-\epsilon_2+\epsilon_3), \quad
\ga_3={1\over 2}(\delta-\epsilon_1+\epsilon_2-\epsilon_3),
\\
\ga_4 &={1\over 2}(\delta-\epsilon_1+\epsilon_2+\epsilon_3), \quad
\ga_5={1\over 2}(\delta+\epsilon_1-\epsilon_2-\epsilon_3).
\end{align*}
In terms of the fundamental weights, we can reexpress the odd roots $\ga_i$ as follows:
\begin{align*}
\gamma_1&={1\over 2}\delta - \omega_3,\quad \gamma_2={1\over 2}\delta - \omega_2+\omega_3,\quad \gamma_3={1\over 2}\delta - \omega_1+\omega_2-\omega_3,\cr
\gamma_4&={1\over 2}\delta - \omega_1+\omega_3,\quad \gamma_5={1\over 2}\delta + \omega_1-\omega_3.
\end{align*}
Besides the conjugate class of the standard simple system $\Pi^0:=\Pi=\{\epsilon_1-\epsilon_2, \epsilon_2-\epsilon_2,\epsilon_3,\ga_1\}$ there are five other conjugate classes of simple systems under the Weyl group action as listed below. They all are obtained via sequences of odd reflections from $\Pi^0$ (cf.  \cite[\S1.4]{CW12}):
\begin{align}
\label{rep of simple roots}
\begin{split}
&\Pi^1=r_{\ga_1}(\Pi^0)=\{\epsilon_1-\epsilon_2,\epsilon_2-\epsilon_3,\ga_2,-\ga_1\},\cr
&\Pi^2=r_{\ga_2}(\Pi^1)=\{\epsilon_1-\epsilon_2,\ga_3,-\ga_2,\epsilon_3\}, \cr
&\Pi^3=r_{\ga_3}(\Pi^2)=\{\ga_5, -\ga_3,\epsilon_2-\epsilon_3,\ga_4\}, \cr
&\Pi^4=r_{\ga_4}(\Pi^3)=\{\delta, \epsilon_3,\epsilon_2-\epsilon_3,-\ga_4\}, \cr
&\Pi^5=r_{\ga_5}(\Pi^3)=\{-\ga_5, \epsilon_1-\epsilon_2,\epsilon_2-\epsilon_3,\delta\}.
\end{split}
\end{align}
Their corresponding Dynkin diagrams are listed as follows:
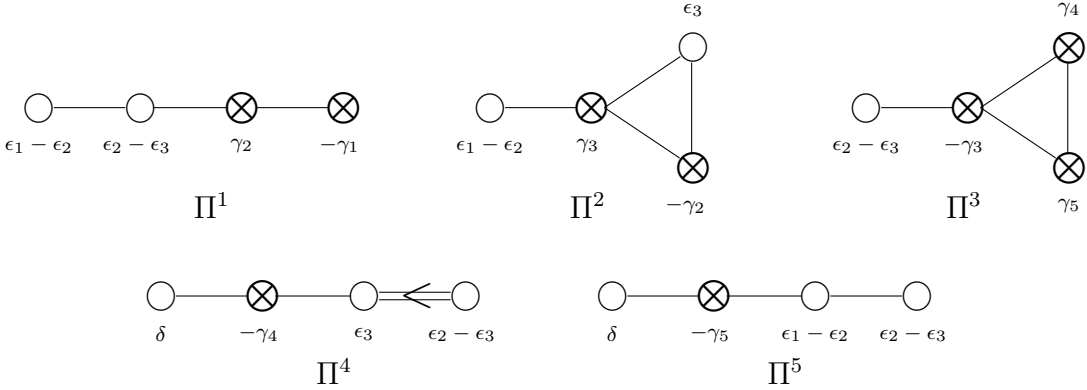
\begin{center}
\begin{tikzpicture}
\node at (0,0.5) {$\bigcirc$};
\draw (0.2,0.5)--(1.15,0.5);
\node at (1.35,0.5) {$\bigcirc$};
\draw (1.52,0.5)--(2.52,0.5);
\node at (2.7,0.5) {$\bigotimes$};
\draw (2.9,0.5)--(3.85,0.5);
\node at (4.05,0.5) {$\bigotimes$};
\node at (0,0) {\tiny $\ep_1-\ep_2$};
\node at (1.3,0) {\tiny $\ep_2-\ep_3$};
\node at (2.7,0) {\tiny $\ga_2$};
\node at (4,0) {\tiny $-\ga_1$};
\node at (2.3,-.8) { $\Pi^1$};
\node at (6,0.5) {$\bigcirc$};
\draw (6.2,0.5)--(7.155,0.5);
\node at (7.35,0.5) {$\bigotimes$};
\draw (7.52,0.5)--(8.53,1.18);
\draw (7.52,0.5)--(8.52,-.17);
\draw (8.68,1.1)--(8.68,-.12);
\node at (8.7,1.3) {$\bigcirc$};
\node at (8.7,-.3) {$\bigotimes$};
\node at (6,0) {\tiny $\ep_1-\ep_2$};
\node at (7.3,0) {\tiny $\ga_3$};
\node at (8.6,-0.8) {\tiny $-\ga_2$};
\node at (8.7,1.8) {\tiny $\ep_3$};
\node at (7.3,-.8) { $\Pi^2$};
\node at (11,0.5) {$\bigcirc$};
\draw (11.2,0.5)--(12.155,0.5);
\node at (12.35,0.5) {$\bigotimes$};
\draw (12.52,0.5)--(13.53,1.18);
\draw (12.52,0.5)--(13.52,-.17);
\draw (13.68,1.1)--(13.68,-.12);
\node at (13.7,1.3) {$\bigotimes$};
\node at (13.7,-.3) {$\bigotimes$};
\node at (11,0) {\tiny $\ep_2-\ep_3$};
\node at (12.3,0) {\tiny $-\ga_3$};
\node at (13.7,-0.8) {\tiny $\ga_5$};
\node at (13.7,1.8) {\tiny $\ga_4$};
\node at (12.3,-.8) { $\Pi^3$};
\end{tikzpicture}
\end{center}
\bigskip

\begin{center}
\begin{tikzpicture}
\node at (0,0.5) {$\bigcirc$};
\draw (0.2,0.5)--(1.15,0.5);
\node at (1.35,0.5) {$\bigotimes$};
\draw (1.52,0.5)--(2.52,0.5);
\node at (2.7,0.5) {$\bigcirc$};
\draw (2.9,0.55)--(3.85,0.55);
\draw (2.9,0.45)--(3.85,0.45);
\node at (4.05,0.5) {$\bigcirc$};
\node at (0,0) {\tiny $\delta$};
\node at (1.3,0) {\tiny $-\ga_4$};
\node at (2.7,0) {\tiny $\ep_3$};
\node at (4,0) {\tiny $\ep_2-\ep_3$};
\node at (3.4,0.5) {\Large $<$};
\node at (2.3,-.5) { $\Pi^4$};
\node at (6,0.5) {$\bigcirc$};
\draw (6.2,0.5)--(7.15,0.5);
\node at (7.35,0.5) {$\bigotimes$};
\draw (7.52,0.5)--(8.52,0.5);
\node at (8.7,0.5) {$\bigcirc$};
\draw (8.9,0.5)--(9.85,0.5);
\node at (10.05,0.5) {$\bigcirc$};
\node at (6,0) {\tiny $\delta$};
\node at (7.3,0) {\tiny $-\ga_5$};
\node at (8.7,0) {\tiny $\ep_1-\ep_2$};
\node at (10,0) {\tiny $\ep_2-\ep_3$};
\node at (8.3,-.5) { $\Pi^5$};
\end{tikzpicture}
\end{center}
The corresponding positive systems are denoted by $\Phi^{i+}$, for $0\le i \le 5$, with  $\Phi^{0+}=\Phi^+$, and the corresponding Borel subalgebras of $\g$ are denoted by $\bbb^i$.

\subsection{Constraints on highest weights}

Let $G$ be the simply connected algebraic supergroup of type $F(3|1)$ whose even subgroup is $SL_2(k) \times \text{Spin}_7(k)$. With respect to the standard Borel subalgebra $\bbb$  (associated to $\Phi^+$), we have
\[
X^+(T) =\{ \la=a\omega_1 +b\omega_2 +c\omega_3 +d\omega_4 \in X~|~a,b,c,d\in \N\}.
\]
Denote the simple $\DG$-module of highest weight $\la$ by $L(\la)$, where $\la \in X^+(T)$.  Assume that the simple module $L(\la) =L^\bbb(\la)$ has $\bbb^{i}$-highest weight $\la^{i}$, for $0\le i\le 5$, where we have set $\la^0=\la, \bbb^0=\bbb$.

\subsubsection{The cases of $d\ge 4$ and $d=0$}

\begin{lem}
  \label{lem:zero}
For any fixed  $0\le i\le 3$,  assume the module $L^{\bbb^i}(\la^{i})$ is finite dimensional and $\la^{i}$ is of the form $(x,y,z,0)$. Let $j=i+1$ if $i\le 2$, and let $j=4$ or $5$ if $i=3$. Then
\[
(\la^i, \ga_j) \equiv 0 \pmod p, \qquad \text{ and }\quad  \la^j =\la^i.
\]
\end{lem}

\begin{proof}
The second equality is an immediate consequence of the first one by Lemma \ref{lem:oddref}.

Assume that $(\la^i, \ga_j) \not \equiv 0.$ Then, by applying the odd reflection $r_{\ga_j}$ and Lemma \ref{lem:oddref}, we have $L^{\bbb^i}(\la^{i}) = L^{\bbb^j}(\la^{j})$, where $\la^{j} = \la^{i} -\ga_j$ is of the form $(*, *, *, -1)$. But then $L^{\bbb^j}(\la^{j})$ cannot be finite dimensional due to the fact $\la^{j} \not \in X^+(T)$, which is a contradiction.
\end{proof}

\begin{prop}\label{prop F4 First}
Let $\la=a\omega_1 +b\omega_2 +c\omega_3 +d\omega_4 \in X^+(T)$.
\begin{itemize}
\item[(1)] If $d\ge 4$, then $L(\la)$ is finite dimensional for arbitrary $a,b,c\in\N$.

\item[(2)] If $d=0$, then  $L(\la)$ is finite dimensional if and only if $a\equiv b\equiv c\equiv 0 \pmod p$.
\end{itemize}
\end{prop}

\begin{proof}
(1)
Let $d\ge 4$. Then $\la+\rho=(a+1,b+1,c+1,d-3)\in X^+(T)$ and it is regular. Hence $L(\la)$ is finite dimensional by Proposition~ \ref{prop:EulerFinite}.

(2)
Assume $L(\la)$ is finite dimensional, for $\la=(a,b,c,0)$.
Lemma~\ref{lem:zero} is applicable and gives us $(\la, \ga_1) \equiv (\la, \ga_2) \equiv (\la,\ga_3) \equiv 0\pmod p$. A direct computation shows
\[
(\la, \ga_1) = -{1\over 2}a-b-{3\over 4}c,
\qquad
(\la, \ga_2) = -{1\over 2}a-b-{1\over 4}c,
\qquad
(\la,\ga_3) = -{1\over 2}a- \frac14 c.
\]
From these we conclude that $a\equiv b\equiv c\equiv 0 \pmod p$.
In this case we have $\la^5=\la^4=\la^3 =\la^2=\la^1=\la$.

By Lemma~\ref{lem:rational}, we see the condition $a\equiv b\equiv c\equiv 0 \pmod p$ is also sufficient for $L(\la)$ to be finite dimensional (this also follows easily by Steinberg tensor product theorem).
\end{proof}

\subsubsection{The case of $d=1$}

\begin{prop}
  \label{prop:d=1simple}
Let $\la=a\omega_1 +b\omega_2 +c\omega_3 +d\omega_4 \in X^+(T)$, with $d=1$. Then  $L(\la)$ is finite dimensional if only if one of the following conditions holds.
\begin{itemize}
\item[(i)] $a\equiv  2b +3\equiv c-1 \equiv 0 \pmod p$;
\item[(ii)] $2a +1\equiv 2b+1 \equiv c\equiv 0 \pmod p$;
    \item[(iii)]  $2a+3\equiv b\equiv c\equiv 0 \pmod p$.
\end{itemize}
\end{prop}

\begin{proof}
Assume $L(\la)$ is finite dimensional, for $\la=(a,b,c,1) \in X^+(T)$. We compute
\[
(\la, \ga_1)=-{1\over 2}a-b-{3\over 4}(c+1).
\]

For now let us assume
$-{1\over 2}a-b-{3\over 4}(c+1)\not\equiv 0 \pmod p$. Then $\la^1=\la-\ga_1=(a,b,c+1,0).$ Hence Lemma~\ref{lem:zero} is applicable and gives us  $(\la^1, \ga_2) \equiv (\la^1,\ga_3) \equiv (\la^1, \ga_4) \equiv 0$. A direct computation shows
\begin{align*}
(\la^1, \ga_2) =-{1\over 2}a-b-{1\over 4}(c+1),
\quad
(\la^1, \ga_3) =-{1\over 2}a-{1\over 4}(c+1),
\quad
(\la^1,\ga_4) =-{1\over 2}a+{1\over 4}(c+1).
\end{align*}
From these we conclude that $a\equiv b \equiv c+1 \equiv 0$. This contradicts the assumption $-{1\over 2}a-b-{3\over 4}(c+1)\not\equiv 0$.

So we always have
\begin{align}
  \label{eq:d=1:laga1=0}
-{1\over 2}a-b-{3\over 4}(c+1) \equiv 0 \pmod p,
\qquad
\text{and}
\qquad
\la^1 =\la=(a,b,c,1).
\end{align}
Using the above equations, we compute
\[
(\la^1,\ga_2)=-{1\over 2}a-b-{1\over 4}c-{3\over 4}\equiv {1\over 2}c \pmod p.
\]
We now divide into 2 cases (1)-(2).

\vskip5pt
(1) Assume $c\not\equiv 0 \pmod p$. Then $\la^2=\la^1-\ga_2=(a,b+1,c-1,0)$; we necessarily have $c\ge 1$. Hence Lemma~\ref{lem:zero} is applicable and gives us that $(\la^2,\ga_3) \equiv (\la^2, \ga_4) \equiv 0$. A direct computation shows
\[
(\la^2, \ga_3) =-{1\over 2}a-{1\over 4}(c-1),
\qquad
(\la^2,\ga_4) =-{1\over 2}a+{1\over 4}(c-1).
\]
From these we conclude $a \equiv c-1\equiv 0$; a revisit of \eqref{eq:d=1:laga1=0} then gives us  $b\equiv -{3\over 2}$.  This gives us Condition~(i) in the proposition. (Note the conditions $c\ge 1$ and \eqref{eq:d=1:laga1=0} are automatically satisfied.) In this case, we have $\la^5=\la^4=\la^3 =\la^2=(a,b+1,c-1,0)$ and $\la^1 =\la$.

\vskip5pt

(2) Assume $c\equiv 0 \pmod p$. So $\la^2=\la^1=\la=(a,b,c,1)$.   We compute
\[
(\la^2,\ga_3)=-{1\over 2}a-{1\over 4}(c+3)\equiv -{1\over 2}a-{3\over 4} \pmod p.
\]
 Now we divide (2) into two subcases (2a)-(2b).
\vskip5pt

\begin{enumerate}
 \item[(2a)]
 Assume $-{1\over 2}a-{3\over 4}\not\equiv 0 \pmod p$. Then $\la^3=\la^2-\ga_3=(a+1,b-1,c+1,0)$; we necessarily have $b\ge 1$. Hence Lemma~\ref{lem:zero} is applicable and gives us that $(\la^3, \ga_4) \equiv 0$. A direct computation shows
 $(\la^3,\ga_4) =-{1\over 2}a +{1\over 4}c -{1\over 4}$. Recalling $c\equiv 0$, we conclude that $a+\hf \equiv 0$. A revisit of \eqref{eq:d=1:laga1=0} then gives us  $b\equiv -\hf$.  This gives us Condition~(ii) in the proposition. (Note the conditions $b\ge 1$ and \eqref{eq:d=1:laga1=0} are automatically satisfied.) In this case, we have $\la^5=\la^4=\la^3=(a+1,b-1,c+1,0)$ and $\la^2=\la^1 =\la$.

 \item[(2b)]
  Assume $-{1\over 2}a-{3\over 4}\equiv 0 \pmod p$. Then $a\equiv -{3\over 2}$, and it follows by $c\equiv 0$ and \eqref{eq:d=1:laga1=0}  that $b\equiv 0$. This gives us Condition~(iii). In this case, we have $\la^4=\la^3=\la^2=\la^1=\la$, and $\la^5 
 =(a-1,b,c+1,0)$.
 \end{enumerate}
 \vskip5pt

 By Lemma~\ref{lem:rational} and by inspection that all weights $\la^i$ lie in $X^+(T)$ for all $i$ in all cases above, we see the conditions (i)-(iii) are sufficient for $L(\la)$ to be finite dimensional. The proposition is proved.
\end{proof}

\subsubsection{The case of $d=2$}

\begin{prop}
  \label{prop:d=2messy}
Assume $\la =a\omega_1 + b\omega_2 +c\omega_3 +d\omega_4 \in X^+(T)$ with $d=2$. Then $L(\la)$ is finite dimensional if only if one of the following conditions hold:
\begin{enumerate}[label*=\arabic*.]
\item
\begin{enumerate}[label*=\arabic*.]
\item $a\equiv c\equiv 0$,  $b\not\equiv -1$ and $b\not\equiv -{3\over 2}$;

\item $a\not\equiv -1$,  $b\equiv -a-1$, $c\equiv 2a$, and $b\ge 1$;

\item $a\not\equiv-{3\over 2}$, $b\equiv 0$ and $c\equiv-2a-4$.
\end{enumerate}
\item
\begin{enumerate}[label*=\arabic*.]

\item $c\equiv 2a+2$,  $b\equiv -2a-3$,  $c\geq 1$, and $a\not\equiv -1$;

\item
   \begin{enumerate}
   \item
    $c\equiv -2a-2$, $b\equiv a$,  $a\not\equiv -{3\over 2}$, $a\not\equiv -1$,  $c\geq 2$, and $a\geq 1$;

    \item
  $a\equiv -{3\over 2}$,  $b\equiv  -{3\over 2}$, $c\equiv 1$;
  \end{enumerate}

   \item $c\equiv 0$, $b\equiv -{1\over 2}a-{3\over 2}$, and $a\not\equiv -3$;

 \item $c\equiv b\equiv 0$ and $a\equiv -3$.
  \end{enumerate}
\end{enumerate}
\end{prop}

\begin{proof}
Assume $L(\la)$ is finite dimensional, for $\la=(a,b,c,2) \in X^+(T)$. We compute $(\la,\ga_1) =-{1\over 2}a-b-{3\over 4}c-{3\over 2}$, and then divide into two cases (1)-(2) below.

(1)
Assume $-{1\over 2}a-b-{3\over 4}c-{3\over 2}\not\equiv 0\pmod p$. Then $\la^1=\la-\ga_1=(a,b,c+1,1)$. We compute $(\la^1,\ga_2) =-{1\over 2}a-b-{1\over 4}c-1$, and then divide into 2 subcases (1a)-(1b).

\vspace{2mm}
(1a)
Assume $-{1\over 2}a-b-{1\over 4}c-1\not\equiv 0\pmod p$. Then $\la^2=\la^1-\ga_2=(a,b+1,c,0)$. Hence Lemma~\ref{lem:zero} is applicable and gives us that $(\la^2,\ga_3) \equiv (\la^2, \ga_4) \equiv (\la^2, \ga_5) \equiv 0$. From these and a direct computation of $(\la^2, \ga_3)=-{1\over 2}a-{1\over 4}c$, $(\la^2, \ga_4)=-{1\over 2}a+{1\over 4}c$, and $(\la^2, \ga_5)={1\over 2}a-{1\over 4}c$, we conclude that $a\equiv c\equiv 0,  b\not\equiv -1, b\not\equiv -{3\over 2}$, whence Condition ~1.1. In this case, we have $\la^5=\la^4 =\la^3=\la^2=(a,b+1,c,0)$.

\vspace{2mm}
(1b)
Assume $-{1\over 2}a-b-{1\over 4}c-1\equiv 0\pmod p$. Then $\la^2=\la^1=(a,b,c+1,1)$. We compute $(\la^2,\ga_3)=-{1\over 2}a-{1\over 4}c-1$, and then again divide into 2 subcases (1b-1)-(1b-2):
\begin{itemize}
  \item[(1b-1)]
  Assume $-{1\over 2}a-{1\over 4}c-{1}\not\equiv 0$. Then $\la^3=\la^2-\ga_3=(a+1,b-1,c+2,0)$. Hence Lemma~\ref{lem:zero} is applicable and gives us that $(\la^3,\ga_4)\equiv (\la^3,\ga_5) \equiv 0$. From these and a direct computation of { $(\la^3, \ga_4)=-{1\over 2}(a+1)+{1\over 4}(c+2)$ and $(\la^3, \ga_5)={1\over 2}(a+1)-{1\over 4}(c+2)$}, we conclude that $c\equiv 2a$. Combining with the conditions on (1), (1b) and (1b-1), this gives us $b\equiv -a-1$ and $a\not\equiv -1$, whence Condition~1.2. In this case we have $\la^2=\la^1=(a,b,c+1,1)$, and $\la^5=\la^4=\la^3=(a+1,b-1,c+2,0)$.

  \item[(1b-2)]
 Assume $-{1\over 2}a-{1\over 4}c-{1}\equiv 0$.  Then $\la^3=\la^2=\la^1=(a,b,c+1,1)$. We deduce from the conditions on (1), (1b) and (1b-2) that  $b\equiv 0,c\equiv-2a-4,a\not\equiv-{3\over 2}$, whence Condition ~1.3. \big(We then compute $(\la^3,\ga_4)=-{1\over 2}a+{1\over 4}c-{1\over 2} \equiv -a-{3\over 2}\not\equiv0$.   Thus, $\la^4=\la^3-\ga_4=(a+1,b,c,0)$. Note that  $(\la^3,\ga_5)={1\over 2}a -{1\over 4} c-1$ ($\equiv a$). Hence $\la^5=\la^3=(a,b,c+1,1)$ if $a\equiv 0$; $\la^5 =(a-1,b,c,0)$ if $a\not\equiv 0$.\big)

\end{itemize}
Case (1b) and hence Case (1) are completed.

\vspace{3mm}

(2)
Assume $-{1\over 2}a-b-{3\over 4}c-{3\over 2}\equiv 0\pmod p$. We have $\la^1=\la=(a,b,c,2)$. Then we compute $(\la^1,\ga_2)=-{1\over 2}a-b-{1\over 4}c-{3\over 2}$, and divide into 2 subcases (2a)-(2b).

\vspace{2mm}
(2a)
Assume $-{1\over 2}a-b-{1\over 4}c-{3\over 2}\not\equiv 0$. Then we compute $\la^2=\la^1-{\ga_2}=(a,b+1,c-1,1)$; we necessarily have $c\geq 1$. (Note the combination of the condition $c\ge 1$ and the condition on (2) implies the condition on (2a).) We further compute $(\la^2,\ga_3)=-{1\over 2}a-{1\over 4}c-{1\over 2}$, and then divide into 2 subcases (2a-1)-(2a-2).
\begin{itemize}
\item[(2a-1)]
 Assume $-{1\over 2}a-{1\over 4}c-{1\over 2}\not\equiv 0$. Then $\la^3=\la^2-\ga_3=(a+1, b,c,0)$. Hence Lemma~\ref{lem:zero} is applicable and gives us that $(\la^3,\ga_4)\equiv (\la^3,\ga_5) \equiv 0$. Combining with the computations of {$(\la^3,\ga_4)=-{1\over 2}(a+1)+{1\over 4}c$ and $(\la^3,\ga_5)={1\over 2}(a+1)-{1\over 4}c$}, this implies $c\equiv 2a+2$ and $b\equiv -2a-3$; moreover the condition on (2a-1) becomes $a\not\equiv -1$. Thus, we have obtained Condition~ 2.1.
  In this case, we have $\la^1=\la$, $\la^2=(a,b+1,c-1,1)$, and $\la^5=\la^4=\la^3=(a+1, b,c,0)$.

\item[(2a-2)]
Assume $-{1\over 2}a-{1\over 4}c-{1\over 2}\equiv 0$. The conditions on (2), (2a) and (2a-2) can be rephrased as $c\equiv -2a-2$, $b\equiv a$ and $a\not\equiv -1$. We have  $\la^3=\la^2=(a,b+1,c-1,1)$; we necessarily have $c\geq 1$. We further compute $(\la^3,\ga_4)=-{1\over 2}a+{1\over 4}c-1\equiv -a-{3\over 2}$, and again divide into 2 subcases:
\begin{itemize}
 \item[(i)]
 Assume $a\not\equiv -{3\over 2}$. Then we have $\la^4=\la^3-\ga_4=(a+1,b+1,c-2,0)$; we necessarily have $c\geq 2$. Moreover, if $(\la^3,\ga_5)=a\not\equiv 0$, then $\la^5=\la^3-\ga_5=(a-1,b+1,c,0)$, requiring $a\geq 1$; otherwise, $\la^5=\la^3$. This gives us Condition~ 2.2(i).

\item[(ii)]
Assume $a\equiv -{3\over 2}$. Then  we have  $b\equiv  -{3\over 2}$ and $c\equiv 1$, whence Condition~ 2.2(ii).
In this case, we have $\la^1=\la$, $\la^4=\la^3=\la^2=(a,b+1,c-1,1)$, and $\la^5=\la^3-\ga_5=(a-1,b+1,c,0)$.
 \end{itemize}
\end{itemize}
This completes Case (2a).

\vspace{2mm}
(2b)
Assume $-{1\over 2}a-b-{1\over 4}c-{3\over 2}\equiv 0\pmod p$. Then $\la^2=\la^1=(a,b,c,2)$. We compute {$(\la^2,\ga_3) =-{1\over 2}a-{1\over 4}c-{3\over 2}$}, and divide into 2 subcases (2b-1)-(2b-2).
\begin{itemize}
\item[(2b-1)]
Assume $-{1\over 2}a-{1\over 4}c-{3\over 2}\not\equiv 0\pmod p$. Then we have $c\equiv 0$, $b\equiv -{1\over 2}a-{3\over 2}$, and $a\not\equiv -3$, whence Condition~ 2.3.
In this case, we have $\la^2=\la^1=\la=(a,b,c,2)$, $\la^3=\la^2-\ga_3=(a+1,b-1,c+1,1)$, and then $(\la^3,\ga_4)=-{1\over2}a+{1\over 4}c-1\equiv -{1\over 2}a-1$ and
$(\la^3,\ga_5)\equiv{1\over2}a-{1\over 2}$. So  $\la^4=\la^3- \ga_4=(a+2,b-1,c,0) $ if $a\not\equiv -2$, and $\la^4=\la^3$ otherwise; moreover,  if $a\not\equiv 1$ then
 $\la^5=\la^3-\ga_5=(a,b-1,c+2,0)$; otherwise $\la^5=\la^3$.

\item[(2b-2)]
Assume $-{1\over 2}a-{1\over 4}c-{3\over 2}\equiv 0\pmod p$. Then we have  $a\equiv -3$, $b\equiv 0$ and $c\equiv 0$, whence Condition~ 2.4.
In this case, we have $\la^i=\la$ for $1\le i \le 5$.
\end{itemize}
Case (2b) and then Case (2) are hence completed. Therefore, we have established the necessary conditions as listed in the proposition for $L(\la)$ to be finite dimensional.

By inspection, we have all weights $\la^i \in X^+(T)$ for all $i$ in every case above. Hence by
Lemma~\ref{lem:rational} we conclude that the conditions as listed in the proposition are also sufficient for $L(\la)$ to be finite dimensional.
 \end{proof}


Now we simplify the above conditions by removing all inequalities. We caution that the resulting conditions are no longer mutually exclusive.

\begin{prop}
  \label{prop:d=2simple}
Set $d=2$. Assume $\la =a\omega_1 + b\omega_2 +c\omega_3 +d\omega_4 \in X^+(T)$.
Then $L(\la)$ is finite dimensional if only if one of the following conditions (i)--(vi) holds:
\begin{enumerate}
\item[(i)] $a\equiv c\equiv 0\pmod p$;

\item[(ii)]  $2a-c \equiv a+b+1\equiv 0\pmod p$; 

 \item[(iii)]  $b\equiv 2a+c+4\equiv 0\pmod p$; 

\item[(iv)]  $2a-c+2 \equiv 2a+ b +3 \equiv 0\pmod p$, and $c\geq 2$;

\item[(v)]   $2a+c+2\equiv a-b\equiv 0\pmod p$, and  $a\geq 1$; 

\item[(vi)]  $a+2b+3 \equiv c\equiv 0\pmod p$.
\end{enumerate}
\end{prop}

\begin{proof}
One first observes that all conditions listed in Proposition~\ref{prop:d=2messy} are part of conditions listed above in this proposition. Indeed the conditions above are basically obtained by removing the inequalities in the conditions in Proposition~\ref{prop:d=2messy}; the cases~(1.4) and (2.4) with no inequalities in Proposition~\ref{prop:d=2messy} are part of (iii) and (vi) above, respectively.

It remains to show that all conditions above in this proposition are included in the list of conditions (1.1)--(1.3) and (2.1)--(2.4) in Proposition~\ref{prop:d=2messy}.

If Condition (i) is satisfied but (1.1) in Proposition~\ref{prop:d=2messy} is not, then either (A) $b\equiv -1$, in which case $a\equiv c\equiv 0$, and so (1.2) is satisfied, or (B) $b \equiv -{3\over 2}$, in which case $a\equiv c\equiv 0$, and so (2.3) is satisfied.

If Condition (ii) is satisfied but (1.2) in Proposition~\ref{prop:d=2messy} is not, then either (A) $a\equiv -1$, in which case $b\equiv0$ and $c\equiv -2$, and hence (1.3) is satisfied,  or (B) $b=0$, in which case, $a\equiv -1, c\equiv -2$, and so (1.3) is satisfied.

If Condition (iii) is satisfied but (1.3) in Proposition~\ref{prop:d=2messy} is not, then $a\equiv -{3\over 2}$, in which case $b\equiv 0, c\equiv -{1}$, and so (2.1) is satisfied.

If Condition (iv) is satisfied but (2.1) in Proposition~\ref{prop:d=2messy} is not, then $a\equiv -1$, in which case $b\equiv -1, c\equiv 0$, and so (2.3) is satisfied.

If Condition (v) is satisfied but (2.2)(i) in Proposition~\ref{prop:d=2messy} is not, then  either (A) $a \equiv -{3\over 2}$, in which case $b \equiv -{3\over 2}, c\equiv 1$, and so (2.2)(ii) is satisfied; or (B) $a \equiv -1$, in which case $b \equiv -1, c\equiv 0$, and so (2.3) is satisfied; or (C) $c=0$, in which case $a\equiv b \equiv -1$, and so (2.3) is satisfied.

If Condition (vi) is satisfied but (2.3) in Proposition~\ref{prop:d=2messy} is not, then $a\equiv -3$, $b \equiv c\equiv 0$, and so (2.4) is satisfied.

The proposition is proved.
\end{proof}

\subsubsection{The case of $d=3$}

\begin{prop}
  \label{prop:d=3messy}
Assume $\la =a\omega_1 + b\omega_2 +c\omega_3 +d\omega_4 \in X^+(T)$, with $d=3$.
Then $L(\la)$ is finite dimensional if only if one of the following conditions holds:

\begin{enumerate}[label*=\arabic*.]
\item
  \begin{enumerate}[label*=\arabic*.]
   \item
     $c\equiv 2a+1$, and $b\not\equiv -2a-{3}$, $b\not\equiv -a-{2}$,  $a\not\equiv -{1}$;
   \item
      $c\equiv -2a-3$, and $b\not\equiv -1$, $b\not\equiv a$;
   \item
     $b\equiv -{1\over 2}a-{1\over 4}c-{7\over 4}$, and $b\not\equiv 0$,  $c \not\equiv -1$;
   \item
     $b\equiv 0$, $c\equiv -2a-7$, and $a\not\equiv -3$.
  \end{enumerate}
\item
  \begin{enumerate}[label*=\arabic*.]
   \item
     $b\equiv -{1\over 2}a-{3\over 4}c-{9\over 4}$, and $c\not\equiv 0$, $c\not\equiv -2a-5$;
   \item
     $b\equiv a+{3\over 2}$,  $c\equiv -2a-5$, and $c\not\equiv 0$;
   \item
     $b\equiv-{1\over 2}a-{9\over 4}$, $c\equiv 0$, and $b\not\equiv 0$;
   \item
     $a\equiv -{9\over 2}$, $b\equiv c\equiv 0$.
\end{enumerate}
\end{enumerate}
\end{prop}

\begin{proof}
Assume $L(\la)$ is finite dimensional, for $\la=(a,b,c,3) \in X^+(T)$. We compute $(\la,\ga_1)= -{1\over 2}a-b-{3\over 4}c-{9\over 4}$, and divide into 2 cases (1)-(2).
\vspace{2mm}

(1)
Assume $-{1\over 2}a-b-{3\over 4}c-{9\over 4}\not\equiv 0\pmod p$. We have $\la^1=\la-\ga_1=(a,b,c+1,2)$. We compute $(\la^1,\ga_2) =-{1\over 2}a-b-{1\over 4}(c+1)-{3\over 2}$, and then divide into 2 cases (1a)-(1b).
\vspace{2mm}

(1a)
Assume $-{1\over 2}a-b-{1\over 4}(c+1)-{3\over 2}\not\equiv 0\pmod p$. Then
$\la^2=\la^1-\ga_2=(a,b+1,c,1)$. We compute $(\la^2,\ga_3)=-{1\over 2}a-{1\over 4}c-{3\over 4}$, and again divide into two subcases (1a-i)-(1a-ii):
\begin{itemize}
\item[(1a-i)]
 Assume $-{1\over 2}a-{1\over 4}c-{3\over 4}\not\equiv 0$. Then $\la^3=\la^2-\ga_3=(a+1,b,c+1,0)$.
 Hence Lemma~\ref{lem:zero} is applicable and gives us that $(\la^3,\ga_4)\equiv (\la^3,\ga_5) \equiv 0$. Combining with the computation of {$(\la^3,\ga_4)=-{1\over 2}(a+1)+{1\over 4}(c+1)$ and $(\la^3,\ga_5)={1\over 2}(a+1)-{1\over 4}(c+1)$}, this implies
 $c\equiv 2a+1$. The conditions on (1), (1a) and (1a-i) can be simplified to $a\not\equiv -{1}$, $b\not\equiv -a-{ 2}$ and $b\not\equiv {-2a}-{3}$, whence Condition ~1.1.
In this case, we have $\la^1=(a,b,c+1,2)$, $\la^2=(a,b+1,c,1)$, and $\la^5=\la^4=\la^3= (a+1,b,c+1,0)$.

\item[(1a-ii)]
Assume $-{1\over 2}a-{1\over 4}c-{3\over 4}\equiv 0$. Then $\la^3=\la^2=(a,b+1,c,1)$. The conditions on (1), (1a) and (1a-ii) can be simplified to $c\equiv -2a-{3}$, $b\not\equiv -1$ and $b\not\equiv a$, whence Condition~ 1.2.
%
In this case, we have $\la^1=(a,b,c+1, 2)$, $\la^3=\la^2=(a,b+1,c,1)$. If $a\not\equiv -{3\over 2}$, then $\la^4=\la^3-\ga_4=(a+1,b+1,c-1,0)$; otherwise, $\la^4=\la^3$. If $a\not\equiv 0$, then $\la^4=\la^3-\ga_5=(a-1,b+1,c+1,0)$; otherwise, $\la^5=\la^3$.
\end{itemize}
This completes Subcase (1a).
\vspace{2mm}

(1b)
Assume $-{1\over 2}a-b-{1\over 4}(c+1)-{3\over 2}\equiv 0\pmod p$. Then $\la^2=\la^1=(a,b,c+1,2)$. We compute $(\la^2,\ga_3)=-{1\over 2}a-{1\over 4}c-{7\over 4}$, and again divide into two subcases (1b-i)-(1b-ii):
\begin{itemize}
 \item[(1b-i)]
  Assume $-{1\over 2}a-{1\over 4}c-{7\over 4}\not\equiv 0$. We compute  $\la^3=\la^2-\ga_3=(a+1,b-1,c+2,1)$. The conditions on (1), (1b) and (1b-i) become $b\equiv -{1\over 2}a-{1\over 4}c-{7\over 4}$, $b\not\equiv 0$, and $c\not\equiv -1$, whence Condition~ 1.3.
In this case, we have $\la^2=\la^1=(a,b,c+1,2)$, and $\la^3=(a+1,b-1,c+2,1)$. Moreover,
if $c\not\equiv 2a+3$, then $\la^4=\la^3-\ga_4=(a+2,b-1,c+1,0)$; otherwise $\la^4=\la^3$. If $c\not\equiv 2a-3$, then $\la^5=\la^3-\ga_5=(a,b-1,c+3,0)$; otherwise $\la^5=\la^3$.

\item[(1b-ii)]
Assume $-{1\over 2}a-{1\over 4}c-{7\over 4}\equiv 0$. Then $\la^3=\la^2=(a,b,c+1,2)$.
The conditions on (1), (1b) and (1b-ii) become $a\not\equiv -3$, $b\equiv 0$ and $c\equiv -2a-7$, whence Condition~ 1.4.
In this case, we have $\la^3=\la^2=\la^1=(a,b,c+1,2)$. Moreover,
If $a\not\equiv -3$, then $\la^4=\la^3-\ga_4=(a+1,b,c,1)$; otherwise $\la^4=\la^3$.
If $a\not\equiv 0$, then $\la^5=\la^3-\ga_5=(a-1,b,c+2,1)$; otherwise $\la^5=\la^3$.
\end{itemize}
This completes Subcase (1b) and then Case (1).
\vspace{2mm}

(2)
Assume  $-{1\over 2}a-b-{3\over 4}c-{9\over 4}\equiv 0$. Then $\la^1=\la=(a,b,c,3)$. We compute $(\la^1,\ga_2)=-{1\over 2}a-b-{1\over 4}c-{9\over 4}$, and divide into 2 subcases (2a)-(2b).

\vspace{2mm}
(2a)
Assume $-{1\over 2}a-b-{1\over 4}c-{9\over 4}\not\equiv 0$.  Then $\la^2=\la^1-\ga_2=(a,b+1,c-1,2)$. We compute $(\la^2,\ga_3)=-{1\over 2}a-{1\over 4}c-{5\over 4}$, and again divide into 2 subcases (2a-i)-(2a-ii):
\begin{itemize}
  \item[(2a-i)]
Assume $-{1\over 2}a-{1\over 4}c-{5\over 4}\not\equiv 0$. Then the conditions on (2), (2a) and (2a-i) become $b\equiv -{1\over 2}a-{3\over 4}c-{9\over 4}$, $c\not\equiv 0$ and $c\not\equiv -2a-5$, whence Condition 2.1.
In this case, we have $\la^1=\la$, $\la^2=(a,,b+1,c-1,2)$, $\la^3=(a+1,b,c,1)$. If $c\not\equiv 2a+5$, then $\la^4= (a+2,b,c-1,0)$; otherwise $\la^4=\la^3$. If $c\not\equiv 2a-1$, then $\la^5= (a,b,c+1,0)$; otherwise $\la^5=\la^3$.

\item[(2a-ii)]
Assume $-{1\over 2}a-{1\over 4}c-{5\over 4}\equiv 0$. Then  the conditions on (2), (2a) and (2a-ii) become $c\equiv -2a-5$, $b\equiv a+{3\over 2}$,  and   $c\not\equiv 0$, whence Condition~ 2.2.
In this case, we have $\la^1=\la$, $\la^3=\la^2=(a,b+1,c-1,2)$. If $a\not\equiv -3$, then $\la^4=(a+1,b+1,c-2,1)$ (and $c\ge 2$ is guaranteed by Condition~2.2); otherwise, $\la^4=\la^3$. If $a\not\equiv 0$, then $\la^5= (a-1,b+1,c,1)$; otherwise, $\la^5=\la^3$.
\end{itemize}
This completes Subcase (2a).

\vskip5pt
(2b)
Assume $-{1\over 2}a-b-{1\over 4}c-{9\over 4}\equiv 0$. Then 
$\la^2=\la^1=\la$. We compute $(\la^2,\ga_3)=-{1\over 2}a-{1\over 4}c-{9\over 4}$, and divide into 2 subcases (2b-i)-(2b-ii):
\begin{itemize}
\item[(2b-i)]
Assume $-{1\over 2}a-{1\over 4}c-{9\over 4}\not\equiv 0$. 
Then $\la^3=\la^2-\ga_3=(a+1,b-1,c+1,2)$.  The conditions on (2), (2b) and (2b-i) become  $c\equiv 0$,  $b\equiv-{1\over 2}a-{9\over 4}$, and $b\not\equiv 0$, whence Condition~ 2.3.
In this case, we have $\la^2=\la^1=\la$, and $\la^3=(a+1,b-1,c+1,2)$. If $a\not\equiv -{7\over 2}$, then $\la^4= (a+2,b-1,c,1)$; otherwise, $\la^4=\la^3$. If $a\not\equiv {5\over 2}$, then $\la^5= (a,b-1,c+2,1)$; otherwise, $\la^5=\la^3$.

\item[(2b-ii)]
Assume $-{1\over 2}a-{1\over 4}c-{9\over 4}\equiv 0$.  The conditions on (2), (2b) and (2b-ii) become $b\equiv 0$,  $c\equiv 0$ and $a\equiv -{9\over 2}$, whence Condition~ 2.4.
In this case, we have $\la^i=\la$ for $1\le i\le 4$ and $\la^5=\la^3-\ga_5=(a-1,b,c+1,2)$.
\end{itemize}
This completes  Case (2). Therefore, we have established  the necessary conditions as listed in the proposition for $L(\la)$ to be finite dimensional.

By inspection, we see that $\la^i\in X^+(T)$ for all $i$ in every case above. By Lemma~\ref{lem:rational}, the conditions listed in the proposition are also sufficient for $L(\la)$ to be finite dimensional.
\end{proof}


Now we simplify the conditions in Proposition~\ref{prop:d=3messy} by removing all inequalities.

\begin{prop}
  \label{prop:d=3simple}
Set $d=3$. Assume $\la =a\omega_1 + b\omega_2 +c\omega_3 +d\omega_4 \in X^+(T)$.
Then $L(\la)$ is finite dimensional if only if one of the following conditions (i)--(v) holds:
\begin{enumerate}
\item[(i)]
     $2a-c +1 \equiv 0\pmod p$;
\item[(ii)]
      $2a+c+3\equiv 0\pmod p$;
\item[(iii)]
     $2a+4b+c+7 \equiv 0\pmod p$;
\item[(iv)]
     $2a+c+7 \equiv b\equiv 0\pmod p$;
\item[(v)]
  $2a +4b +3c +9\equiv 0\pmod p$.
\end{enumerate}
\end{prop}

\begin{proof}
One first observes that all conditions listed in Proposition~\ref{prop:d=3messy} are part of conditions listed above in this proposition. Indeed the conditions above are basically obtained by removing the inequalities in the conditions in Proposition~\ref{prop:d=3messy}, and the case (2.4) with equalities only is included in (v).

It remains to show that all conditions above in this proposition are included in the list of conditions (1.1)--(1.4) and (2.1)--(2.4) in Proposition~\ref{prop:d=3messy}.

We first check that the 4 subcases (2.1)--(2.4) of Proposition~\ref{prop:d=3messy} are equivalent to Condition ~(v).
If Condition (v) is satisfied but (2.1) of Proposition~\ref{prop:d=3messy} is not, then  we have 2 cases (A)-(B):
\begin{itemize}
\item[(A)]  $c\not \equiv 0$ and
$c\equiv -2a-5$, in which case $b\equiv a+{3\over 2}$,  $a\not\equiv -{5\over 2}$, and so (2.2) is satisfied;
\item[(B)] $c\equiv 0$. Then $b\equiv-{1\over 2}a-{9\over 4}$. We further divide into 2 subcases:
 \begin{itemize}
  \item[(B1)] $b\not\equiv 0$, in which case (2.3) is satisfied,
  \item[(B2)] $b\equiv 0$, in which case $a\equiv -{9\over 2}$, and so (2.4) is satisfied.
 \end{itemize}
\end{itemize}

 If Condition (i) is satisfied but (1.1) of Proposition~\ref{prop:d=3messy} is not, then we have the following 3 cases (A)-(B)-(C):
\begin{itemize}
\item[(A)] $b \equiv -2a-3$, in which case $c \equiv 2a+1$, and so (v) is satisfied;
\item[(B)] $a\equiv -1$ and $b \not\equiv -2a-3$, in which case $c\equiv -1$ but $b \not \equiv -1$, and so (1.2) is satisfied;
\item[(C)] $b\equiv -a -2$ and $a\not \equiv -1$. Hence $c\not\equiv -1$ thanks to $c \equiv 2a+1$. We further divide into 2 subcases below:
 \begin{itemize}
  \item[(C1)] $b\not\equiv 0$, in which case $c\equiv 2a+1$, and so (1.3) is satisfied,
  \item[(C2)] $b\equiv 0$, in which case $a\equiv -2, c\equiv -3$, and so (1.4) is satisfied.
  \end{itemize}
\end{itemize}

 If Condition (ii) is satisfied but (1.2) of Proposition~\ref{prop:d=3messy} is not, then either
(A)  $b\equiv a$, in which case $c\equiv -2a-3$, and so (v) is satisfied,
or (B) $b\equiv -1$ and $b\not \equiv a$, in which case $c\equiv -2a-3$ and then $c\not\equiv -1$, and  so (1.3) is satisfied.

 If Condition (iii) is satisfied but (1.3) of Proposition~\ref{prop:d=3messy} is not, then either (A) $c \equiv -1$, in which case the equality $b\equiv -{1\over 2}a-{1\over 4}c-{7\over 4}$ implies that (v) is satisfied; or (B) $c\not \equiv -1$ and $b\equiv 0$, in which case $a\not \equiv -3$, and so (1.4) is satisfied.

 If Condition (iv) is satisfied but (1.4) of Proposition~\ref{prop:d=3messy} is not, then $a\equiv -3$, in which case $b\equiv 0, c\equiv -1$, and so (v) is satisfied.

 The proposition is proved.
\end{proof}

\subsection{Simple modules of the supergroup $F(3|1)$}

Summarizing Propositions~\ref{prop F4 First}, \ref{prop:d=1simple}, \ref{prop:d=2simple}, and \ref{prop:d=3simple}, we have established the following classification of simple modules for type $F(3|1)$.

\begin{thm}  \label{thm:F4}
Let $p>3$. Let $G$ be the simply connected supergroup of type $F(3|1)$. A complete list of inequivalent simple $G$-modules consists of $L(\la)$, where
$\la = a \om_1 +b \omega_2 +c \omega_3 +d \frac{\delta}{2}$, with $a,b,c,d \in \N$, such that
one of the following conditions is satisfied:
\begin{enumerate}
\item $d=0$, and $a\equiv b\equiv c\equiv 0 \pmod p$.

\item
$d=1$, and $a,b,c$ satisfy either of (i)-(iii) below:
\begin{itemize}
\item[(i)] $a\equiv  2b +3\equiv c-1 \equiv 0\pmod p$;
\item[(ii)] $2a +1\equiv 2b+1 \equiv c\equiv 0\pmod p$;
    \item[(iii)]  $2a+3\equiv b\equiv c\equiv 0\pmod p$.
\end{itemize}

\item
$d=2$, and $a,b,c$ satisfy either of (i)-(vi) below:
\begin{enumerate}
\item[(i)] $a\equiv c\equiv 0\pmod p$;

\item[(ii)]  $2a-c \equiv a+b+1\equiv 0\pmod p$;

 \item[(iii)]  $b\equiv 2a+c+4\equiv 0\pmod p$;
\item[(iv)]  $2a-c+2 \equiv 2a+ b +3 \equiv 0\pmod p$ and $c\geq 2$;

\item[(v)]   $2a+c+2\equiv a-b\equiv 0\pmod p$ and  $a\geq 1$;

\item[(vi)]  $a+2b+3 \equiv c\equiv 0\pmod p$.
\end{enumerate}

\item
$d=3$, and $a,b,c$ satisfy either of (i)-(v) below:
\begin{enumerate}
\item[(i)]
     $2a-c +1 \equiv 0\pmod p$;
\item[(ii)]
      $2a+c+3\equiv 0\pmod p$;
\item[(iii)]
     $2a+4b+c+7 \equiv 0\pmod p$;
\item[(iv)]
     $2a+c+7 \equiv b\equiv 0\pmod p$;
\item[(v)]
  $2a +4b +3c +9\equiv 0\pmod p$.
\end{enumerate}

\item $d\ge 4$, (and $a,b,c\in \N$ are arbitrary).
\end{enumerate}
\end{thm}
We do no attempt the classification of simple $G$-modules for $p=3$ in this paper, and leave it to the reader.

\begin{rem}
Theorem~\ref{thm:F4} makes sense over $\C$, providing an odd reflection approach to the classification of finite-dimensional simple modules over $\C$ (due to \cite{Kac77}; also cf. \cite{Ma14}). 
Indeed this classification can be read off from Theorem~\ref{thm:F4} (by regarding $p=\infty$) as follows.
{\it The simple $\g$-modules $L(\la)$ over $\C$ are finite dimensional if and only if $\la= a \om_1 +b \omega_2 +c \omega_3 +d \frac{\delta}{2}$, for $a,b,c,d \in \N$, satisfies one of the 3 conditions: (1) $a=b=c=d=0$; (2) $d=2$ and $a=c=0$; (3) $d\ge 4$. }
\end{rem}


\end{document}